\documentclass[11pt]{amsart}
\usepackage{amsfonts,amsmath,amsthm,mathrsfs}

\usepackage{graphics}
\usepackage{graphicx,color}
\usepackage{amsmath,amsfonts,amssymb,amsopn,amscd,amsthm}
\usepackage{mathrsfs}
\usepackage{bookmark}
\usepackage{breqn}
\usepackage{mathtools}
\usepackage[normalem]{ulem}

\let\saveboldsymbol\boldsymbol
\usepackage{bm}
\let\boldsymbol\saveboldsymbol

\newtheorem{thm}{Theorem}[section]
\newtheorem{cor}[thm]{Corollary}
\newtheorem{lem}[thm]{Lemma}{\rm}

{\rm}

\newtheorem{rem}[thm]{Remark}
\newtheorem{ex}{Example}
\numberwithin{equation}{section}

\def\x{\mathbf{x}}

\def\R{\mathbb{R}}
\def\T{\mathbf{T}}
\def\N{\mathbb{N}}

\def\Q{\mathbf{Q}}

\def\M{\mathbf{M}}

\def\A{\mathbf{A}}

\def\B{\mathbf{B}}

\def\Y{\mathbf{Y}}
\def\X{\mathbf{X}}

\def\v{\mathbf{v}}
\def\y{\mathbf{y}}

\def\bx{\mathbf{x}}

\def\y{\mathbf{y}}

\def\a{\mathbf{a}}

\def\i{\mathbf{i}}
\def\j{\mathbf{j}}

\def\balpha{\boldsymbol{\alpha}}
\def\bbeta{\boldsymbol{\beta}}

\def\X{\mathbf{X}}

\def\dis{\displaystyle}

\def\Pb{\mathbb{P}}

\def\x{\mathbf{x}}
\def\v{\mathbf{v}}
\def\y{\mathbf{y}}
\def\a{\mathbf{a}}

\def\A{\mathbf{A}}
\def\M{\mathbf{M}}
\def\Q{\mathbf{Q}}

\def\R{\mathbb{R}}
\def\N{\mathbb{N}}

\def\bla{\boldsymbol{\lambda}}
\def\bom{\boldsymbol{\Omega}}
\def\dis{\displaystyle}
\definecolor{dkgreen}{rgb}{0,0.4,0}
\definecolor{dkred}{rgb}{0.7,0,0}

\newcommand{\app}{\hat f}

\DeclareMathOperator{\Argmin}{Argmin}

\usepackage[normalem]{ulem}

\thanks{The research of M. Korda and J . B. Lasserre research is partly supported by AI Interdisciplinary Institute ANITI funding, through the French ``Investing for the Future PIA3'' program under the Grant agreement n$^{\circ}$ANR-19-PI3A-0004. 
This research is also part of the programme DesCartes and is supported by the National Research Foundation, Prime Minister's Office, Singapore under its Campus for Research Excellence and Technological Enterprise (CREATE) programme and by the European Union under the project ROBOPROX (reg. no. CZ.02.01.01/00/22 008/0004590). J.B. Lasserre also acknowledges support from ANR-NuSCAP-20-CE48-0014.}

\begin{document}

\author{Didier Henrion$^{1,2}$, Milan Korda$^{1,2}$, Jean Bernard Lasserre$^{1,3}$}

\footnotetext[1]{CNRS; LAAS; Universit\'e de Toulouse, 7 avenue du colonel Roche, F-31400 Toulouse, France. }
\footnotetext[2]{Faculty of Electrical Engineering, Czech Technical University in Prague, Technick\'a 2, CZ-16626 Prague, Czechia.}
\footnotetext[3]{Institute of Mathematics; Universit\'e de Toulouse, 118 route de Narbonne, F-31062 Toulouse, France. }


\title[Polynomial argmin for recovery and approximation]{Polynomial argmin for recovery and approximation of multivariate discontinuous functions}

\begin{abstract}
We propose to approximate a (possibly discontinuous) multivariate function $f(\x)$ on a bounded set by the {partial minimizer} $\arg\min_{y} p(\x,y) $ of an
appropriate polynomial $p$ whose construction can be cast in a \emph{univariate} sum of squares (SOS) framework, resulting in a highly structured convex semidefinite program.
In a number of non-trivial cases (e.g. when $f$ is a piecewise polynomial)
we prove that the approximation is exact with
a low-degree polynomial $p$. Our approach has three distinguishing features: 
(i) It is mesh-free and does not require the knowledge of the discontinuity locations. 
(ii) It is model-free in the sense that we only assume that the function to be approximated is available through samples (point evaluations). (iii) The size of the semidefinite program is independent of the ambient dimension and depends linearly on the number of samples.
We also analyze the sample complexity of the approach, proving a generalization error bound in a probabilistic setting. This allows for a comparison with machine learning approaches. 
\end{abstract}

\maketitle

\section{Introduction}

Approximation of discontinuous functions in multiple dimensions is a notoriously difficult problem and a scientific  challenge. 
A common strategy (e.g. described in \cite{tadmor}) which works well in the univariate setting (and is implemented for example in the {\tt chebfun} package \cite{chebfun}) 
consists of the following steps :
1) detect the discontinuity locations and split the domain into a disjoint union of regions where the function is continuous, and 2) construct 
approximations of the continuous pieces on each region. However, in the multivariate case this strategy is very challenging to implement since the discontinuity set may have a positive dimension (see, e.g, \cite{bozzini} where an algorithm for detecting discontinuities in two dimensions is proposed). 

{For instance, in the
\emph{Variable Scaled Discontinuous Kernel} (VSDK) method described e.g. in
\cite{kernel}
the authors provide interesting numerical experiments 
but the method indeed assumes prior knowledge of the set of discontinuities\footnote{Indeed, in \cite[p. 442]{kernel}
the authors state
\emph{``The only drawback of the procedure lies in the fact that the algorithm needs to know where
the discontinuities occur".}}. In addition, the complexity of the approximant increases with the the sample size. 
On the other hand, \emph{Weigthed Essentially Non-Oscillatory} (WENO) methods as described in e.g. \cite{WENO} (and initially
designed for hyperbolic PDEs) do not assume prior knowledge of discontinuities.
However, such methods construct \emph{local} models based on an increasing number of samples (at every candidate input point of the algorithm for approximation) and so do not provide a global model of the discontinuous function to approximate.}
Numerical difficulties faced with approximating
multivariate functions are illustrated in the example sections of the paper.

A typical and important application is concerned with classification in data analysis and supervised 
learning, where powerful deep learning methods have obtained impressive results and success stories. However, such powerful methods still have some limitations (even for learning continuous functions, 
let alone discontinuous functions). 
Indeed for instance and quoting \cite{siam-news-2}, \emph{``
Despite many results that establish the existence of Neural Nets (NNs) with excellent approximation properties, 
algorithms that can compute these NNs only exist in specific cases."} That is, no training algorithm can obtain them in the general case.
For an interesting discussion about such limits (instability, accuracy, etc.) the interested reader is referred to 
\cite{siam-news-1,siam-news-2} and references therein. For learning discontinuous functions by neural networks, \cite{bernardo} proposes a tailored architecture; however this approach requires knowledge of discontinuity locations and is limited to univariate problems.

{This paper is a follow-up (but non trivial extension) of \cite{constructive-1}. We provide an alternative approximation technique aimed at dealing with such discontinuities and the Gibbs phenomenon, which are large oscillations of the approximation near the discontinuity points, see e.g. \cite[Chapter 9]{atap}. In particular, we show that our class of approximants can model exactly multivariate piecewise polynomial functions and can approximate with arbitrary accuracy other discontinuous functions.}

\subsection*{Contribution}
We introduce a new class of approximants for a possibly discontinuous function $f$ from $\X \subset \R^n$ to $\Y \subset \R$. We propose to approximate $f$ by
the {{\it polynomial argmin}}
\begin{equation}
\label{arg-intro}
\x\mapsto \app(\x)\,:={\min\{}\arg\min_{y\in \Y}p(\x,y)\}\,,\quad \x\in\X\,,
\end{equation}
{where $\Y\supset f(\X)$ and $p\in\R[\x,y]$ is a polynomial in $(\x,y)$.}

The main features of our approach can be summarized as follows.
\begin{itemize}
    \item The approach is mesh-free and does not require the knowledge of the discontinuities locations.
    \item 
    {It} is not limited to  univariate functions or tensor products thereof.
    \item 
    {It} is model-free, working only with the samples of the unknown function.
    \item A polynomial $p$ for our approximant \eqref{arg-intro} is constructed using  {a very specific class of convex optimization (semidefinite programming) problems whose size depends (linearly) on the number of samples and is \emph{independent} of the ambient dimension.}
    \item The approximant is simple to evaluate as it is the argmin of a \emph{univariate} polynomial.
    \item We provide a generalization error analysis in a probabilistic setting.
\end{itemize}

{In addition we also provide a result which is interesting in its own and justifies the use of the argmin approximation. Namely, we prove that any piecewise (possibly discontinuous) polynomial function $f:\X\to\R$ on a bounded set $\X\subset\R^d$ has an \emph{exact} ``$\arg\min$" representation. Specifically, provided that the partition that defines the regions of continuity of $f$ is of a certain kind, there exists a polynomial $p$ (not unique in general) such that
\[f(\x)\,=\,\displaystyle\arg\min_{y\in \R}\,\,p(\x,y)\]
for all $\x\in\X$, except for points of discontinuity of $f$. In other words, $\hat{f}$ in \eqref{arg-intro} coincides with~$f$ (except for discontinuity points). 
Moreover and importantly, when $f$ is known only from a sample of its values, such a polynomial $p$ can be retrieved by our algorithm, as $p$ belongs to its set of 
possible optimal outputs; see Remark \ref{rem:optimality}.}

We believe that these features make the approach a {unique and 
promising tool} with a wide variety of applications in data analysis. This is corroborated by a numerical evidence where we observe a remarkable performance on a range of examples. We also provide a solid theoretical underpinning of the method but leave some questions open, including the optimal rate of convergence of the argmin approximant.

{
\noindent{\bf Prior work and novelty.}
The idea behind our new approximant is 
a non-trivial extension of  our previous work \cite{constructive-1} 
that has provided an approximant which is the argument of the partial minimum (argmin), of a sum of squares (SOS) of polynomials,
in fact the reciprocal of the Christoffel function of a measure, ideally supported on the {graph} of the function to recover.
This recovery procedure can be seen as a non-standard application of the Christoffel-Darboux kernel.
Importantly, being in a class of functions much larger than polynomials,
such an approximant is able to approximate some discontinuous functions much better
than  polynomials can. 
Remarkably, in non-trivial examples, recovery is possible without oscillations and Gibbs phenomenon usually encountered 
in several more standard approaches. In this respect the reader is referred to 
the detailed discussion in \cite{poly-kernel} on kernel variants {(e.g F\'ejer or Jackson kernels) to attenuate} the Gibbs phenomenon encountered with polynomial approximations.}

While in \cite{burgers} our original motivation for introducing the polynomial argmin approximant was to recover the solution of a nonlinear partial differential equation from the knowledge of its approximate moments, this strategy was made rigorous and generalized to graph recovery from moments in \cite{constructive-1}. Then it was later extended to cope with partial moment information \cite{constructive-2}. Following our initial work, this argmin strategy was called 
``implicit model'' and used in robotics applications \cite{google} where the polynomial $p$ in  \eqref{arg-intro} is replaced by a continuous function computed by training, e.g. a neural network.

{A novelty and distinguishing feature of the present paper with respect to \cite{constructive-1} is that
the polynomial $p$ in \eqref{arg-intro} is \emph{not} restricted to be the reciprocal of the Christoffel function associated with the measure supported on the graph of $f$. Indeed, our set of potential candidate polynomials $p$
is now a {suitably chosen subspace of} $\R[\x,y]$, which is much larger than the set considered in \cite{constructive-1}. In addition, from a computational perspective, our method has a much more favorable behavior with respect to the ambient dimension. While 
in \cite{constructive-1} the size of the moment matrix whose Cholesky factorization is used to construct the Christoffel-Darboux kernel grows rapidly with respect to the ambient dimension, the size of the semidefinite programs (SDPs) solved in this work is independent of the dimension.}

\subsection*{Outline}

In the motivational Section \ref{sec:motivation} we show that our polynomial argmin strategy is already efficient in some non-trivial cases. For instance, it allows {\it exact recovery} when $f$ is a polynomial, or
an algebraic function, or
a (possibly discontinuous) piecewise polynomial. 

However, exact recovery by a polynomial argmin cannot be guaranteed in general,
and in Section \ref{sec:sos} we provide a numerical scheme to obtain the polynomial $p$ which appears in~\eqref{arg-intro}, 
when knowledge on $f$ is only though its finitely many values on a sample of points 
$(\x(i))_{i\in I}\subset\X$ {(and without any a priori knowledge on its distribution).} We reformulate our polynomial argmin strategy in the framework of univariate sum of squares (SOS) positivity certificates
so that finding $p$ amounts to solving a semidefinite optimization problem whose size is controlled by the degree  of {$p$ in $y$} and the sample size.

%
%
%
%
Finally, in Section \ref{sec:examples}
we also provide a set of numerical experiments  to evaluate 
{the efficiency of our proposed polynomial argmin strategy 
on a sample of problems.}  We observe that it
performs remarkably well to approximate challenging discontinuous functions already tested in  \cite{constructive-1}, 
as well as non trival two-dimensional examples of functions whose
set of discontinuities has positive dimension.
We have also compared with machine learning methods based on neural networks. {In all our experiments, the proposed method achieves far better accuracy with  simpler representation of the approximant.}

\section{{Exact representations}}\label{sec:motivation}

The purpose of this section is to demonstrate the expressive power of the polynomial argmin. We do so by showing that for several classes of known functions, an exact (and often simple) ``argmin" representation
is possible. Subsequently, in Section~\ref{sec:sos}, we show how a polynomial $\arg\min$ approximation can be found using convex optimization, provided that a collection of finite samples of values of $f$ is available. 

\subsection*{Notation and definitions}

Let $\R[\x,y]$ denote the ring of polynomials in the variables $\x\in\R^n$ and $y\in\R$,
and $\R[\x,y]_d$ its subset of polynomials of total degree at most $d$.
A polynomial $p$ is a sum of squares (SOS) if it can be written as $\sum_{k} p_k^2$
for finitely many polynomials $p_k$. The convex cone of all SOS polynomials of degree at most $d$ in the variable $\x$ is denoted by $\Sigma_d[\x]$.

\subsection{Representation of Polynomials}\label{sec:polynomials}

When $f \in \R[\bx]$ is a given polynomial, in~\eqref{arg-intro} choose  \[
(\x,y)\mapsto p(\bx,y):=  \frac{1}{2}y^2 - f(x)y\,,\quad\forall \x,y\,.
\]
Indeed, we observe that $\frac{d p}{dy} = y - f(x)$ and hence $y = f(x)$ is a stationary point. Since $p$ is strictly convex in $y$, it follows that $y = f(x)$ is the  global minimizer. The degree of $p$ in $x$ is equal to the degree of $f$ whereas the degree in $y$ is equal to two irrespective of $f$.

\subsection{Representation of Algebraic functions}

An {\it algebraic function} $f$ is such that $q_k(\bx,f(\bx))=0$ for some given polynomials $q_k \in \R[\bx,y]$, $k=1,\ldots,m$. 

If for each $\bx$, $(\bx,f(\bx))$ is the unique common zero of the $q_k$ then 
choose 
\[(\x,y)\mapsto p(\bx,y):=\sum_{k=1}^m q_k(\bx,y)^2\,,\quad\forall \x,y\,,\] and observe that the degree of $p$ in $\x$ and $y$ is twice the maximal degree of the $q_k$ in these variables.

Note that {\it semi-algebraic functions}\footnote{A semi-algebraic function is such that its graph is described by a finite union of a finite intersection of sets defined by polynomial equations and inequalities.} can also be modeled like that, provided that the inequalities are incorporated in the definitions of the domain $\X$ and image sets $Y$.

\begin{ex}\label{ex:abs}
The absolute value function can be expressed as
\[
|x|\:=\:\arg\min_{y \in Y} (x^2-y^2)^2
\]
with $\Y:=[0,1]$, for all $x \in \X:=[-1,1]$.
Note that a more complicated degree 8 polynomial argmin model for the absolute value function was already described in \cite[Example 2]{constructive-1}.
\end{ex}

 {
\subsection{Representation of piecewise constant functions}\label{sec:piecewise_const}

Let $\X \subset \R^n$ be bounded and $\Y = \R$. Given $N$ polynomials $g_1,\ldots,g_N \in \R[\x]$, define the cells 
\begin{equation}
    \label{model-1}
\X_i = \{\x \in \X \mid g_i(\x) \,<\, g_j(\x)\,, \:\forall j\neq i\},\quad i=1,\ldots,N\,,
\end{equation}
that partition $\X$. This model is motivated by the proof technique used and at the same time is sufficiently general to cover many situations encountered in practice, including box partitions and partitions where each cell is defined by a sublevel set of a single polynomial. We now briefly discuss these partitions and then detail the argmin representation of piecewice constant and piecewice polynomial functions defined thereon.}

Consider first $\X = [-1,1]$ split into $N$ intervals $\X_i = (a_i,a_{i+1})$, $i = 1,\ldots,N$, with $a_i < a_{i+1}$, $a_1 = -1$ and $a_{N+1} = 1$. Then $g_i(x) = (x-a_i)(x-a_{i+1})$ satisfies $g_i(x) < g_j(x)$ for all $j\ne i$ if and only if $x \in \X_i$.

This example is readily generalized to $\X = [-1,1]^n$ split into axis-aligned boxes $\X_\i = \bigtimes_{k=1}^n (a^{k}_{i_k}, a^{k}_{i_k+1})$, where $a_{k,1},\ldots,a_{k,m_k}$ (with $m_k\in\mathbb{N}$, $a_{k,1}=-1$, $a_{k,m_k}=1$) define the breakpoints of each coordinate axis $k\in \{1,\ldots,n\}$ and $\i = (i_1,\ldots,i_n)$ is a multi-index with $i_k \in \{1,\ldots,m_k\}$ . In that case, we can define $g_\i(\x) = \sum_{k=1}^n (x_k-a^k_{i_k})(x_k-a^k_{i_k+1})$. Then $g_\i(\x) < g_\j(\x)$ for all multi-indices $\j\ne \i$ if and only if $\x \in \X_\i$.

Finally, partitions of $\X$ into disjoint cells $\X_i$
of the form 
\[\X_i \,=\, \{\,\x \mid g_i(\x) < 0\,\}\,,\quad i=1,\ldots,N\,,\]
that is,
defined by the sublevel set of a single polynomial $g_i$,
can also be readily modeled using~(\ref{model-1}). Indeed, since $\X_i$'s are disjoint, we have $\x\in \X_i$ if and only if $g_i(\x) < 0$ and $g_j(\x) >0$ for all $j\ne i$; hence also $g_i(\x) < g_j(\x)$ for all $j\ne i$. Cartesian products of cells of this kind can be modeled using the same trick as described for boxes (i.e., cartesian products of intervals). 


Now we discuss the argmin representation of piecewise constant functions defined on the partition~(\ref{model-1}); generalization to piecewice polynomials is in \S \ref{sec:piece-poly}.  Given $N$ distinct\footnote{\label{foot:merge}{If the constants $c_i$ are not all distinct, then the cells associated with each subset of indices sharing the same value must be merged, thereby producing a new (coarser) partition with distinct values on each cell.}} real numbers $c_1,\ldots,c_N$, the piecewise constant function $f$ is defined by $f(\x) = c_i$ whenever $\x \in \X_i$. The value of $f$ at the boundary points of the partition, i.e., when $g_i(\x) = g_j(\x)$ for some $i\ne j$ is left arbitrary as in these regions one cannot hope for exact representation using the argmin.

\begin{thm}\label{thm:piecewise_cosnt}
    Let $f:\X\to \R$, $\X$ bounded, be a piecewise constant function satisfying $f(\x) = c_i$ for  $\x \in \X_i$ with $c_1,\ldots,c_N$ distinct. Then there exists $p\in\R[\bx,y]$ such that
\[f(\bx) = \displaystyle\arg\min_{y\in \R} p(\bx,y) \text{\;\;for all\;\;} \x \in \bigcup_{i=1}^N \X_i.
\]

\end{thm}
\begin{proof}
    Let \[
L_i(y)
=\prod_{\substack{1\le j\le n\\j\neq i}}
\frac{y-c_j}{\,c_i-c_j\,}
\]
be the degree-\((n-1)\) Lagrange  interpolation polynomial at \(c_i\). Define the degree-\((2n-1)\) Hermite–cardinal polynomial
\[
H_i(y) = 
\bigl(1 - 2\,L_i'(c_i)\,(y-c_i)\bigr)\;\bigl[L_i(y)\bigr]^2.
\]
This polynomial satisfies $H_i(c_j) = 1$ if $i=j$ and $H_i(c_j) = 0$ if $i\ne j$. Furthermore $H'_i(c_j)= 0$ for all $i,j$. Define 
\[
p(\x,y) = \sum_{i=1}^N g_i(\x) H_i(y) + M\prod_{i=1}^N (y-c_i)^2,
\]
where $M$ is a sufficiently large constant ensuring that $p(\x,y)$ is coercive in $y$ for every $\x \in \X$. Such constant exists since $\X$ is bounded. The coercive penalty ensures that for every $\x\in\X$, the global minimum of $p(\x,\cdot)$ is attained at a critical point of $p(\x,\cdot)$.

Now, since $H_i'(c_j) = 0$ for all $i,j$, we have $(\partial p/\partial y)(\x,y) = 0$ if $y \in \{c_1,\ldots, c_N\}$. Therefore, for each $\x$, the constants $c_1,\ldots,c_N$ are the only critical points of $p(\x,\cdot)$. The coercive penalty ensures that $c_1,\ldots,c_N$ are strict local minima for sufficiently large $M$. Furthermore, since $H_i(c_j) = 1$ if $i=j$ and zero  otherwise, we have $p(\x,c_i) = g_i(\x)$.  Therefore, among the local minima $c_1,\ldots,c_N$, the one with the lowest value of $g_i(\x)$ attains the smallest value of $p$ as desired. There are $N-1$ additional critical points $\xi_i$ of $p(\x,y)$, each lying in the interval $(c_i,c_{i+1})$ (assuming without loss that $c_i$'s are ordered). These additional critical points are simple roots of $(\partial p/\partial y)(\x,\cdot)$. Since $\xi_i\in (c_i,c_{i+1})$ and each $c_i$ is a strict local minimum, each $\xi_i$ must be a strict local maximum and hence not a candidate for a minimizer of $p(\x,\cdot)$. This concludes the proof.
\end{proof}

\subsection{Piecewise polynomial functions}
\label{sec:piece-poly}

{Theorem~\ref{thm:piecewise_cosnt} can be generalized to the case of piecewise polynomial functions defined over the same partition simply by replacing the constants $c_i$ by polynomials $h_i(\x)$. The only caveat is  that the constants $c_i$ appear in the denominator and hence the resulting function is rational. However, given that the argmin of a function is not changed when multiplied by a positive number, we can consider the polynomial 
\[
\tilde p(\x,y) =
 p(\x,\y) \prod_{\substack{1\le i,j\le n\\j\neq i}}
(h_i(\x)-h_j(\x))^4,
\]
where all polynomials appearing in the denominator of $p$ are canceled. Therefore $\tilde p$ is a polynomial in $(\x,y)$. Hence we arrive at the following corollary:

\begin{cor}\label{cor:piecewise_cosnt}
    Let $\X\subset\R^n$ be bounded, and let $\X_i$ be as in \eqref{model-1}, 
    $i=1,\ldots,N$. Let $f:\X\to \R$ be a piecewise polynomial function satisfying $f(\x) = h_i(\x)$ for  $\x \in \X_i$ with $h_1,\ldots,h_N \in \R[\x]$. Then there exists $\tilde p\in\R[\bx,y]$ such that
\[f(\bx) = \displaystyle\arg\min_{y\in \R} \tilde p(\bx,y) \text{\;\;for all\;\;} \x \in \bigcup_{i=1}^N \X_i \setminus \A, 
\]
where $\A = \{\x \mid h_i(\x) = h_j(\x)\; \text{for some\;} i\ne j\}$.
\end{cor}
We note that the removal of the points $\x$ where $h_i(\x) = h_j(\x)$ cannot be easily avoided since for those $\x$, the polynomial $\tilde p$ is identically zero. However, the Lebesgue measure of $\A$ is zero, unless $h_i(\x) = h_j(\x)$ for all $\x$, in which case cell merging can be carried out (see Footnote~\ref{foot:merge}).

The construction via the proof of Theorem \ref{thm:piecewise_cosnt} (and hence of Corollary \ref
{cor:piecewise_cosnt})
is not unique. For instance in the case of two pieces, let $g, p_1, p_2 \in \R[\bx]$ and consider the piecewise polynomial function
\begin{equation}
\label{model-2}
\bx\mapsto f(\bx) := \left\{\begin{array}{lcl}
p_1(\bx) & \text{if} & g(\bx) \in (0,1) \\
p_2(\bx) & \text{if} & g(\bx) \in (-1,0)
\end{array}\right.\,,\quad \bx\in\X\,,
\end{equation}
where $\X:=\{\x:g(\x) \in (-1,0)\}\cup\{\x: g(\x) \in (0,1)\}$.

\begin{lem}\label{pwpoly}
Let $f$ and $\X$ be as in \eqref{model-2}. Then there exists $p\in\R[\bx,y]$ such that
$f(\bx)=\displaystyle\arg\min_{y\in \Y} p(\bx,y)$ for all $\bx\in \X$, with $\Y=\R$.
\end{lem}
\begin{proof}
With $r\in\R[\bx]$ and $q\in\R[\bx,y]$, let
$p(\bx,y) := (y-p_1(\bx))^2(y-p_2(\bx))^2+r(\bx)q(\bx,y)$ be such that $\partial q(\bx,y)/\partial y=6(y-p_1(\bx))(y-p_2(\bx))$.
For instance $q(\bx,y) :=  2y^3-3(p_1(\bx)+p_2(\bx))y^2+6p_1(\bx)p_2(\bx)y$.

Now observe that for each given $\bx \in \X$, $y \mapsto p(\x,y)$ is a coercive quartic univariate polynomial, and hence it has at most two local minima. 
Letting $r(\bx):=g(\bx)\,(p_1(\bx)-p_2(\bx))$,
the gradient $\partial p(\x,y)/\partial y$ vanishes at the critical points $p_1(\x)$ resp. $p_2(\x)$ resp. $p_3(\x):=(p_1(\x)(1-3g(\x))+p_2(\x)(1+3g(\x))/2$. At the critical points, the Hessian $\partial^2 p(\x,y)/\partial y^2$ is equal to $2(p_1(\x)-p_2(\x))^2(1+3g(\x))$ resp.
$2(p_1(\x)-p_2(\x))^2(1-3g(\x))$ resp. $(p_1(\x)-p_2(\x))^2(-1+3g(\x))(1+3g(\x))$. 

To compare the values at the critical points, we evaluate the differences $p(\x,p_2(\x))-p(\x,p_1(\x))=g(\x)(p_1(\x)-p_2(\x))^4$, $p(\x,p_3(\x))-p(\x,p_2(\x))=(1+g(\x))(1-3g(\x))^3(p_1(\x)-p_2(\x))^4/16$, $p(\x,p_3(\x))-p(\x,p_1(\x))=(-1+g(\x))(1+3g(\x))^3(p_1(\x)-p_2(\x))^4/16$. The proof follows by evaluating the signs of the Hessian and the differences at the critical points for $g(\x)$ within the intervals $(-\infty,-1)$, $(-1,-1/3)$, $(-1/3,0)$, $(0,1/3)$, $(1/3,1)$, $(1,\infty)$, see Table \ref{tab:signs}. 
\end{proof}
}

\begin{table}[ht]
    \centering
    \begin{tabular}{c|cccccc}
        & $(-\infty,-1)$ & $(-1,-1/3)$ & $(-1/3,0)$ & $(0,1/3)$ & $(1/3,1)$ & $(1,\infty)$ \\ \hline
        $p_1$ & max & max & min & min & min & min \\
        $p_2$ & min & min & max & max & max & max \\
        $p_3$ & min & min & max & max & min & min \\ \hline
        argmin & $p_3$ & $p_2$ & $p_2$ & $p_1$ & $p_1$ & $p_3$   \end{tabular}
    \caption{Nature of critical points for different intervals of variation of $g(\x)$.
    \label{tab:signs}}
\end{table}

\begin{figure}[ht]
    \centering
    \includegraphics[width=0.6\textwidth]{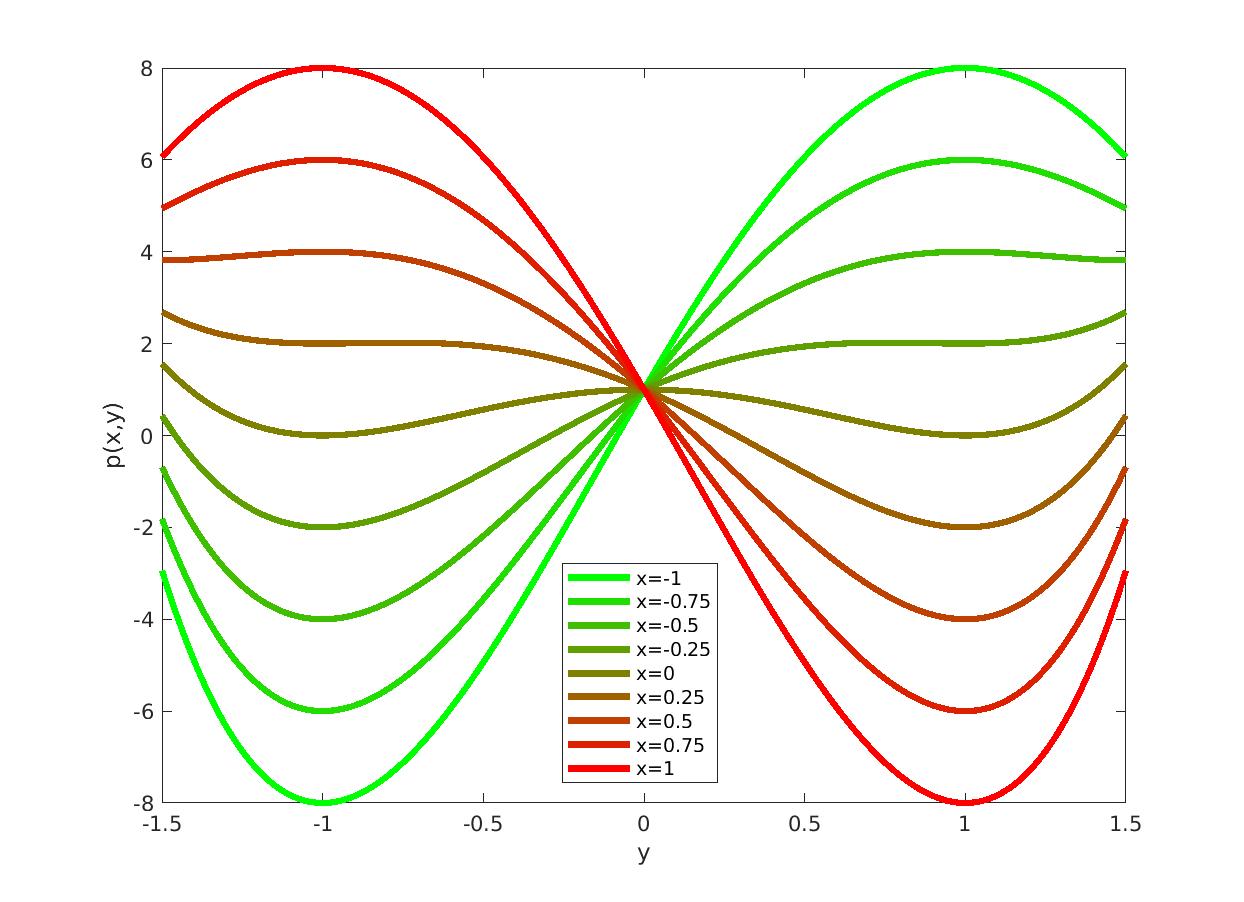}
    \caption{Degree 4 polynomial $p(x,y)$ whose argmin w.r.t. $y$ models the sign function. Represented are univariate polynomials $y \mapsto p(x,y)$ for various given values of $x$. We observe that for positive values of $x$, the argmin is precisely $+1$ whereas for negative values of $x$, the argmin is precisely $-1$ as required in order to represent the function $\mathrm{sign}(x)$.
    \label{fig:argminsign}}
\end{figure}

\begin{ex}\label{ex:sign}
In \cite[Example 1]{constructive-1} a degree 8 polynomial argmin model was described for the sign function $f(x)$ which is equal to $-1$ if $x\in[-1,0)$ and $+1$ if $x\in(0,1]$. A simpler representation is obtained via Lemma \ref{pwpoly} with $p(x,y)=(y+1)^2(y-1)^2+4xy(y^2-3)$. Here we consider $\Y = \R$ and $\X = [-1,1]$.  Figure~\ref{fig:argminsign} depicts the polynomial $p(x,y)$ for different values of $x$.


An even simpler argmin model of the sign function on $\X=[-1,1]$ is $p(x,y)=-xy$ with $\Y=[-1,1]$. Similarly to Example \ref{ex:abs}, the choice of domain and image sets $\X, \Y$ plays here a key role.
\end{ex}

\begin{ex}\label{ex:disk} The indicator function of the bivariate unit disk {$\{\x \in \R^2 : x^2_1+x^2_2\leq 1\}$} can be modeled exactly on {$\X:=\{\x \in \R^2 : x^2_1+x^2_2\leq 2\}$} via Lemma \ref{pwpoly}, as the degree 5 polynomial argmin of $p(\bx,y)=y^2((y-1)^2+(1-x^2_1-x^2_2)(3-2y))$
\end{ex}
{
\begin{ex}
    \label{ex-relu}
    The univariate function
	\[
	x\mapsto f(x)\;=\;\max\{x,0\}\,,\quad x\,\in\,(-1,1)\,,
	\]
	often called the \emph{rectified linear unit} or ReLU, is widely used in optimization and machine learning.  One may ask whether there is a \emph{polynomial} \(p(x,y)\) for which \(f(x)\) arises as
	\(\arg\min_y p(x,y)\). 

By Lemma \ref{pwpoly}, and with $p_1(x)=x$, $p_2(x)=0$, $g(x)=x$, one can check that $p(x,y) = y^2((y-x)^2+(2y-3x)x^2)$ is such that $\arg\min_{y \in \mathbb R} p(x,y) = \max(0,x)$ for all $x \in (-1,1)$.
\end{ex}
}

\section{{Sample-based formulation}}\label{sec:sos}

{Before describing our numerical scheme we first state a classical result 
on certificate of non-negativity in real algebraic geometry, 
which plays a key role in our approach.}

\subsection{Nonnegativity of univariate polynomials on an interval}
The following result whose second part is due to F. Luk\'acs is classical (see, e.g., \cite[p. 4681]{powers} for a discussion).

{\begin{thm}\label{thm:unisos}
A univariate polynomial $q \in \R[y]$ of degree $d$ is nonnegative on $\mathbb{R}$ if and only if $q\in \Sigma_d[y]$.
A univariate polynomial $q \in \R[y]$ of degree $d$ is nonnegative on the interval $[a,b] \subset \R$ if and only if
\[
\begin{cases}
q = \sigma_0 + \sigma_1(b-y)(y-a),\quad \sigma_0\in \Sigma_d[y], \;\;\;\;\,\sigma_1 \in \Sigma_{d-2}[y] & d \mathrm{\;\, even,} \\
q = \sigma_0(y-a) + \sigma_1(b-y),\quad \sigma_0\in \Sigma_{d-1}[y], \;\sigma_1 \in \Sigma_{d-1}[y] & d \mathrm{\;\, odd.}
\end{cases}
\]
\end{thm}
}
It is a simple observation that a polynomial $\sigma$ belongs to $\Sigma_d$ (for $d$ even) if and only if there exists a positive semi-definite matrix $W\succeq 0$ such that $\sigma(y) = v_{d/2}(y) W v_{d/2}(y) $, where $v_{d/2}$ is a basis of $\R[y]_{d/2}$, e.g., $v_{d/2} = [1,y,y^2,\ldots,y^{d/2}]$. Therefore the nonnegativity of $q \in \R[y]$ on $\mathbb{R}$ or $[a,b]$ is \emph{equivalent} to the feasibility of a semidefinite programming (SDP) problem\footnote{Semidefinite Programming (SDP) is a (convex) conic optimization problem on the space of positive
semidefinite matrices, a generalization of Linear Programming; several efficient software packages exist (e.g. MOSEK \cite{MOSEK} or SeDuMi \cite{SeDuMi}), and for more details, the interested reader is referred to e.g. \cite{sdp-handbook} and references therein.}.
Observing that the conditions in Theorem~\ref{thm:unisos} are affine in the coefficients of $q$, we conclude that one can also \emph{optimize} over the set of polynomials (of fixed degree) nonnegative over $[a,b]$ using { Semidefinite Programming; see e.g. \cite{lasserre01,acta}.}

\subsection{A numerical scheme}

Given a function $f:\R^n\to\Y$ with $\Y= [a,b] \subset \R$ sampled at points $(\x_i,y_i)_{i=1}^N$ with $y_i = f(\x_i)$, we wish to construct an approximation $\hat{f}_d$ of the form
\begin{equation}
\label{approx}
\x\mapsto \app(\x)\,:=\,\arg\min_{y\in \Y }p(\x,y),
\end{equation}
\noindent where $p\in\R[\x,y]$ is a polynomial to be determined. {The set $\mathbf{Y}\subset\mathbb{R}$ serves as a priori information on the range of $f$. If no such information is available, we set $\mathbf{Y} = \mathbb{R}$ (see Remark~\ref{rem:range} for more details).}

Throughout this section we assume that the argmin is unique. If this is not the case, a tiebreaker rule has to be applied (e.g., one can consider the min of the argmin). {Since  monomials not containing $y$ do not influence the argmin}, we use the parametrization of $p$ as
\begin{eqnarray}\label{eq:hkparam}
(\x,y)\mapsto p(\x,y) \,:=\, \sum_{k=1}^{d_y} h_k(\x)y^k\,,\quad\forall \x,y\,,
\end{eqnarray}
where $h_k \in \R[\x]_{d_x}$ are polynomials of total degree at most $d_x$ to be determined. In order to find $h_k$, we propose to solve the following convex optimization problem parametrized by $d_x,d_y,d_\gamma \in \mathbb{N}$ and $\alpha > 0$:

\begin{equation} \label{opt:pos_data}
\begin{aligned}
 \min\limits_{\gamma \in \R[\x,y]_{d_\gamma},\, (h_k \in \R[\x]_{d_x})_{k=1}^{d_y}} \quad & \int_{\X\times [a,b]}\gamma(x,y)\,d\x dy \\
\textrm{s.t.} \quad &   p(\x_i,y) - p(\x_i,y_i) + \gamma(\x_i,y_i) - \alpha(y-y_i)^2  \ge 0 \,, \forall y\in [a,b],\;\; i=1,\ldots,N\\
   \quad & \gamma(\x_i,y_i) \ge 0\,,\quad i=1,\ldots,N,
\end{aligned}
\end{equation}
where $p$ is parametrized using $h_k$ as in~(\ref{eq:hkparam}). The rationale behind~(\ref{opt:pos_data}) is as follows: If $\gamma(\x_i,y_i) = 0$, then $p(\x_i,y) - p(\x_i,y_i) \ge  \alpha(y-y_i)^2$ for all $y \in \Y = [a,b]$ and hence $y_i = \arg\min_{y\in\Y} p(\x_i,y)$, which means $\app(\x_i) = y_i$. Therefore, if $\gamma(\x_i,y_i) = 0$ for all $i=1,\ldots,N$ we get an exact interpolation of all data points. This, however, cannot be achieved in general in which case $\gamma(\x_i,y_i) > 0$ for some $i$; the polynomial $\gamma$ therefore acts as a slack variable and is minimized in the objective function. An alternative to using a polynomial slack variable is to assign one slack variable $\gamma_i\in \R_+$ to each of the constraints; here we chose to use the polynomial slack variable in order to make the number of decision variables of~(\ref{opt:pos_data}) independent of the number of samples $N$, which facilitates the analysis of the generalization error in Section~\ref{sec:generalization}. 

We also note that if the objective function cannot be evaluated in closed form (e.g., if the moments of the Lebesgue measure over $\X\times [a,b]$ are not known or too costly to compute), the integral can be replaced by an approximation computed from the available data
\begin{equation}\label{eq:emp_average}
\frac{1}{N}\sum_{i=1}^N \gamma(\x_i,y_i). 
\end{equation}

The optimization problem~\eqref{opt:pos_data} translates to 
{a semidefinite program (SDP)} by equivalently reformulating the first constraint using Theorem~\ref{thm:unisos}, {as is done 
in the Moment-SOS hierarchy of convex relaxations for polynomial optimization. For more details the interested reader is referred to \cite{lasserre01,acta}.}

For brevity, we state it explicitly only for $d_y$ even (the difference for $d_y$ odd is the same as in Theorem~\ref{thm:unisos}):
\begin{equation} \label{opt:sos}
\begin{aligned}
 \min\limits_{\gamma \in \R[\x,y]_{d_\gamma},\, (h_k\in\R[\x]_{d_x})_{k=1}^{d_y},\, (\sigma_{0,i}, \, \sigma_{1,i})_{i=1}^N } \quad & \int_{\X\times [a,b]}\gamma(\x,y)\,d\x dy \\
 \quad &   \hspace{-3.5cm}p(\x_i,y) - p(\x_i,y_i) + \gamma(\x_i,y_i) - \alpha(y-y_i)^2 = \sigma_{0,i} + \sigma_{1,i}(b-y)(y-a),\;\; i=1,\ldots,N\\
   \quad & \hspace{-3.5cm} \sigma_{0,i} \in \Sigma_{d_y}[y], \sigma_{1,i} \in \Sigma_{d_y-2}[y]\\
   \quad & \hspace{-3.5cm}\gamma(\x_i,y_i) \ge 0,\quad i=1,\ldots,N.
\end{aligned}
\end{equation}

{Formulation (\ref{opt:sos}) readily translates to an SDP 
and can be solved using off-the-shelf solvers such as MOSEK \cite{MOSEK} or SeDuMi \cite{SeDuMi}.}

{
\begin{rem}[Simplified complexity analysis]
In  SOS problem \eqref{opt:sos}, if we neglect the number of free variables parametrizing $h_k$ and $\gamma$ (polynomials of degrees typically much smaller than $N$), we have $2N$ positive semidefinite matrices of size at most $d_y/2$ and satisfying $Nd_y$ equality constraints. In this simplified setup, according to \cite[Section 6.6.3]{bn01}, the number of Newton steps of an interior-point algorithm for finding an $\varepsilon$-solution of SOS problem \eqref{opt:sos} is of the order of $\log(1/\varepsilon)\sqrt{N}d_y$, and each Newton step has a complexity of the order of $N^3d^4_y+N^3d^3_y+N^2d^4_y$.
\end{rem}}

{
\begin{rem}[Range of $f$]\label{rem:range}
If no information on the range of $f$ is available, the first constraint of~\eqref{opt:sos} can be replaced by 
\[
p(\x_i,y) - p(\x_i,y_i) + \gamma(\x_i,y_i) - \alpha(y-y_i)^2 = \sigma_{0,i}\quad i=1,\ldots,N,
\]
with $\sigma_{0,i} \in \Sigma_{d_y}$. This is equivalent to $p(\x_i,y) - p(\x_i,y_i) + \gamma(\x_i,y_i) - \alpha(y-y_i)^2$ being nonnegative on $\R$ for each $i = 1,\ldots,N$.
\end{rem}
}

{We note that, up to rescaling of $p$ and $\gamma$, the problems  \eqref{opt:pos_data} and\eqref{opt:sos} are invariant with respect to the choice of the parameter $\alpha > 0$. We decided to include this parameter as a simple way to control the scaling of the coefficients of $p$ and $\gamma$ in the numerical implementation.}
{
\begin{rem}\label{rem:optimality}
For piecewice constant and piecewice polynomial functions described in \S\,\ref{sec:piecewise_const} and \S\,\ref{sec:piece-poly}, the polynomials $p$ providing their argmin representation constructed  in Theorem~\ref{thm:piecewise_cosnt} and Corollary~\ref{cor:piecewise_cosnt} are optimal in (\ref{opt:pos_data}) after rescaling (which does not change the argmin). This holds as long as the samples $\x_i$ lie outside of the points of discontinuity of $f$ (which is of zero Lebesgue measure). In other words, our blind data-driven approach is able to recover exactly optimal solutions for a large class of functions for which a Gibbs  phenomenon would occur with more classical approaches. To prove the optimality of $p$, it suffices to observe that by construction $(\partial^2 p / \partial y^2) (\x_i,y_i) > 0$ if all $\x_i$'s lie outside the points of discontinuity of $f$. Since there are finitely many data points, we can rescale $p$ such that $(\partial^2 p / \partial y^2) (\x_i,y_i) > 2\alpha$ for all $i\in 1,\ldots,N$. Since $(\partial p / \partial y) (\x_i,y_i) = 0$, we get by Taylor's theorem $p(\x_i,y) - p(\x_i,y_i) \ge \alpha(y-y_i)^2$. This estimate holds globally since $y_i$ is the unique global minimizer of $p(\x_i,\cdot)$ and $p(\x_i,\cdot)$ is coercive. Therefore $p$ satisfies the constraints of (\ref{opt:pos_data}) with $\gamma = 0$ and is therefore optimal since $\gamma$ is nonnegative.

\end{rem}
 }



{
\begin{rem}
    We remark that with $\alpha > 0$ the argmin in~(\ref{approx}) is unique on a given data point $x_i$ provided that $\gamma(x_i,y_i) = 0$ (i.e., there is a zero error on the data point). Outside of the data set, we cannot guarantee the uniqueness of the argmin in which case a tiebreaker rule must be applied (e.g., taking the min of the argmin). Importantly, the results below are independent of whether the argmin is unique or not since they apply to the entire set of minimizers.
\end{rem}
}

The following result bounds the error on data points of any feasible solution to (\ref{opt:pos_data}) and (\ref{opt:sos}).

\begin{lem}[Error on data points]\label{lem:error_on_data}
Let $p$, $\gamma$ and $\alpha$ be feasible in~(\ref{opt:pos_data}) or (\ref{opt:sos}) and let
\[
z_i \in \arg\min_{y\in[a,b]} p(\x_i,y).
\]
Then we have
\[
 |z_i - y_i| \le \sqrt{\frac{\gamma(\x_i,y_i)}{\alpha}}
\]
and therefore 
\[
|\app(\x_i) - f(\x_i)|\le \sqrt{\frac{\gamma(\x_i,y_i)}{\alpha}}
\]
for $i=1,\ldots, N$.
\end{lem}
\begin{proof}
Let $(\x_i,y_i)$ be fixed. Observe that for any $z_i \in \arg\min_{y\in[a,b]} p(\x_i,y)$, we have $p(\x_i,z_i) - p(\x_i,y_i) \le 0$. Therefore we have by the first constraint of~(\ref{opt:pos_data}) or (\ref{opt:sos})
\[
\gamma(\x_i,y_i) - \alpha(z_i-y_i)^2 \ge p(\x_i,z_i) - p(\x_i,y_i) + \gamma(\x_i,y_i) - \alpha(z_i-y_i)^2 \ge 0.
\]
Therefore 
\[
|z_i - y_i| \le \sqrt{\frac{\gamma(\x_i,y_i)}{\alpha}}.
\]
The last statement of the lemma follows by the facts that $\app(\x_i) \in \arg\min_{y\in[a,b]}p(\x_i,y)$ with $p$ optimal in~(\ref{opt:pos_data}) and (\ref{opt:sos}) and $y_i = f(\x_i)$.
\end{proof}

\subsection{Generalization error}\label{sec:generalization}
In this section we study the generalization error of the argmin estimator in a probabilistic setting. We assume that the samples $\x_i$, $i=1,\ldots,N$, are independent identically distributed, drawn from a probability distribution $\Pb$ on $\X$ that is unknown to us. We study the generalization error using the tools of scenario optimization, which allows for analysis with minimal underlying assumptions on $f$. The generalization bounds obtained have no explicit dependence on the dimension of the ambient space $n$ and on regularity of $f$. They depend only the number of decision variables in~(\ref{opt:pos_data}), which, however, may depend implicitly on $n$.

We first observe that Problem~\ref{opt:pos_data} can be equivalently rewritten in the form
\begin{equation} \label{opt:generalization_form}
\begin{aligned}
 \min\limits_{\theta \in \R^{n_\theta}} \quad & c^\top\theta \\
\textrm{s.t.} \quad &   \inf_{y\in[a,b]}\{ p_\theta(\x_i,y) - p_\theta(\x_i,y_i) + \gamma_\theta(\x_i,y_i) - \alpha(y-y_i)^2\}  \ge 0 ,\;\; i=1,\ldots,N\\
   \quad & \gamma(\x_i,y_i) \ge 0,\quad i=1,\ldots,N,
 \end{aligned}
\end{equation}
where $\theta \in \R^{n_\theta}$ gathers all the decision variables of~(\ref{opt:pos_data}), i.e., the coefficients of $(h_k)_{k=1}^{d_y}$ and $\gamma$, and $c \in \R^{n_\theta}$ is a constant vector such that $c^\top \theta = \int \gamma\, d\x dy$. 

Problem~(\ref{opt:generalization_form}) can be rewritten as
\begin{equation} \label{opt:scenario_sample}
\begin{aligned}
 \min\limits_{\theta \in \R^{n_\theta}} \quad & \theta^\top c \\
\textrm{s.t.} \quad &  \theta \in \bigcap_{i=1}^N\Theta_i,
\end{aligned}
\end{equation}
where the set $\Theta_i$ is defined by
\[
\Theta_i := \Big\{\theta \mid \inf_{y\in[a,b]} \{p_\theta(\x_i,y) - p_\theta(\x_i,y_i) + \gamma_\theta(\x_i,y_i) - \alpha(y-y_i)^2  \ge 0\}, \; \gamma_\theta(\x_i,y_i) \ge 0
\Big\}.
\]
We also define 
\[
\Theta_\x := \Big\{\theta \mid \inf_{y\in[a,b]}\{p_\theta(\x,y) - p_\theta(\x,f(\x)) + \gamma_\theta(\x,f(\x)) - \alpha(y-f(\x))^2  \ge 0\} ,\; \gamma_\theta(\x,f(\x)) \ge 0 \Big\}.
\]
We observe that~(\ref{opt:scenario_sample}) is the so-called scenario counterpart of the robust optimization problem
\begin{equation} \label{opt:robust}
\begin{aligned}
 \min\limits_{\theta \in \R^{n_\theta}} \quad & \theta^\top c \\
\textrm{s.t.} \quad &  \theta \in \bigcap_{\x\in\X}\Theta_\x.
\end{aligned}
\end{equation}
In other words, the feasible set in~(\ref{opt:scenario_sample}) is a sub-sampled version of the feasible set of~(\ref{opt:robust}), where only $N$ constraints, drawn independently, are enforced. Crucially, we remark that both $\Theta_i$ and $\Theta_\x$ are convex and so are the feasible sets of~(\ref{opt:scenario_sample}) and (\ref{opt:robust}). With these notations, we can state the following theorem:
\begin{thm}\label{thm:generalization}
Let $n_\theta$ denote the number of decision variables in~(\ref{opt:scenario_sample}). Suppose that $\epsilon \in (0,1)$, $\delta \in (0,1)$ and $N \in \N$ are chosen such that
\begin{equation}\label{eq:bound}
\sum_{k=0}^{n_{\theta}-1}
{N\choose k}
\epsilon^k(1-\epsilon)^{N-k}
\le \delta
\end{equation}
or 
\begin{equation}\label{eq:bound_sufficient}
N \ge \frac{1}{\epsilon}\left(n_\theta-1 + \ln\frac{1}{\delta} + \sqrt{2(n_\theta-1)\ln\frac{1}{\delta}}  \right)
\end{equation}
hold ((\ref{eq:bound_sufficient}) is a sufficient condition for~(\ref{eq:bound})).
Let $(p,\gamma)$ be an optimal solution to~(\ref{opt:sos}) with $N$ iid samples from a probability distribution $\Pb$ on $\X$ and denote \[
\gamma_{\mathrm{max}} := \sup_{(\x,y)\in\X\times [a,b]}\gamma(\x,y).
\]
Then with probability at least $1-\delta$ (taken over the sample $\x_1,\ldots,\x_N$ with joint distribution $\underbrace{\Pb\otimes\ldots\otimes\Pb}_{N\;\mathrm{times}}$, we have
\begin{equation}\label{eq:bound_main}
\Pb\Big( \big\{\x \in \X : |\app(\x) - f(\x)| \le \sqrt{\frac{\gamma_{\mathrm{max}}}{\alpha}} \big\}   \Big) > 1 - \epsilon,
\end{equation}
where $\app(\x)$ is any measurable selection satisfying $\app(\x) \in \Argmin\limits_{y\in \Y} p(\x,y)$. 
\end{thm}
\begin{rem}
We remark that if the points $\x_i$ are sampled uniformly in $\X$, the probability in~(\ref{eq:bound_main}) is nothing but the normalized Lebesgue measure on $\X$.
\end{rem}

{
\begin{rem}[Parameter choice]
Theorem~\ref{thm:generalization} can be used as a guiding tool for selecting the parameters $d_x,d_y$. Specifically, these can be increased incrementally until a desired accuracy measured by $\gamma_{\mathrm{max}}$ is obtained. This is especially appealing when the degree of $\gamma$ is zero (i.e., $\gamma$ is a real number) in which case the value of $\gamma_{\mathrm{max}}$ is readily available. Otherwise, this value can be upper-bounded using the moment-sum-of-squares hierarchy~\cite{lasserre01} or other methods providing rigorous bounds on the range of a polynomial.
\end{rem}
}

\begin{proof}[Proof of Theorem~\ref{thm:generalization}.]
Any optimal solution $(p,\gamma)$ of~(\ref{opt:sos}) induces an optimal solution $\theta$ to~(\ref{opt:scenario_sample}) and vice-versa. Given such an optimal solution, \cite[Theorem~1]{campi} asserts that if~(\ref{eq:bound}) holds, then with probability at least $1-\delta$
\[
\Pb(  \{\x \mid \theta \not\in \Theta_\x \}  ) \le \epsilon
\]
or equivalently
\[
\Pb(  \{\x \mid \theta \in \Theta_\x \}  ) > 1-\epsilon.
\]
Using the definition of $\Theta_\x$, this implies that
\[
 \Pb\big(A \big) > 1-\epsilon,
 \]
 where
 \[
 A := \{\x\in \X \mid \inf_{y\in[a,b]}\{p(\x,y) - p(\x,f(\x)) + \gamma(\x,f(\x)) - \alpha(y-f(\x))^2  \ge 0 \}.
 \]
Given any $\x \in A$ and taking $\app(\x) \in \arg\min_{y\in[a,b]}p(\x,y)$, we have
\[
\gamma(\x,f(\x)) - \alpha(\app(\x) - f(\x))^2 \ge p(\x,\app(\x)) - p(\x,f(\x))  + \gamma(\x,f(\x)) - \alpha(\app(\x) - f(\x))^2 \ge 0,
\]
where we used the fact that $p(\x,\app(\x)) - p(\x,f(\x)) \le 0$, $\app(\x) \in [a,b]$ and that $\x \in A$. It follows that
\[
(\app(\x) - f(\x))^2 \le \frac{\gamma(\x,f(\x))}{\alpha} \le \frac{\gamma_{\mathrm{max}}}{\alpha}
\]
for all $\x \in A$. Taking square roots and recalling that $P(A) > 1-\epsilon$, we obtain the result. The condition~(\ref{eq:bound_sufficient}) is a sufficient condition for~(\ref{eq:bound}) derived in~\cite[Corollary~1]{roberto}.
\end{proof}
\begin{rem}
When the integral in the objective function~\ref{opt:pos_data} is replasserced by its empirical average~\eqref{eq:emp_average} computed from the same data set that is used to enforce the constraints, it is currently unknown whether the generalization bound of Theorem~\ref{thm:generalization} remains valid~\cite{campi_personal}. A simple remedy is to use an independent sample for the objective and for the constraints, for example by splitting the data set in two.
\end{rem}

\subsection{Beyond polynomials}
{
In this section we briefly discuss how the proposed method extends to approximants of the form
\[
\app(\x) = \arg\min_{y \in \Y} p(\x,y),
\]
where $p$ is not necessarily a polynomial.  The key observation to make is that when parametrizing $p$ as
\begin{eqnarray}\label{eq:hkparam-2}
p(\x,y) = \sum_{k=1}^{d_y} h_k(\x)y^k,
\end{eqnarray}
the functions $h_k$ appear in~\eqref{opt:pos_data} and \eqref{opt:sos} only via their evaluations at the data points $\x_i$. Therefore, parametrizing each $h_k$ as $h_{k}(\x) = \sum_{i=1}^{n_k} c_{k,i} \beta_{k,i}(\x)$, where $\beta_{k,i}$ are possibly non-polynomial basis functions, the optimization problem~(\ref{opt:sos}) remains a semidefinite programming problem with the decision variables $c_{i,k} \in \R$ and the coefficients of $\gamma$. The function $\gamma$ can be parametrized by non-polynomial basis functions in $(\x,y)$ since only evaluations of $\gamma$ at the samples $(\x_i,y_i)$ appear in~(\ref{opt:sos}).

If a non-polynomial parametrization of $p$ in $y$ was sought, one would have to resort to certificates of nonnegativity for the given function classerss akin to Theorem~\ref{thm:unisos}. Currently, this is well understood for trigonometric polynomials~\cite{dumitrescu} but we envision broader function classes may be considered, given the univariate nature of the nonnegativity certificate required in~(\ref{opt:sos}) which is significantly less challenging than its multivariate counterpart.
}

\section{Numerical examples}\label{sec:examples}

In this section we present several examples demonstrating the effectiveness of the polynomial argmin data structure for regression of functions possessing discontinuities. All examples were solved on a MacBook Air 1.2 GHz Quad-Core Intel Core i7 with 16GB RAM, MOSEK SDP solver \cite{MOSEK}. The problems were modeled using Yalmip~\cite{yalmip}. The range of all functions was normalized to $\Y = [-1,1]$ and the parameter $\alpha$ was taken to be 0.01 in all examples. The parameters that vary in the examples are $d_x$ and $d_y$ in the parametrization of $p$ in~\eqref{eq:hkparam} (i.e., $d_x$ and $d_y$ are the degrees of $p$ in $\x$ and $y$ respectively).
{Matlab prototype codes reproducing the numerical experiments can be downloaded from \url{https://homepages.laas.fr/henrion/software/polyargmin}}

\subsection{Univariate: Approximation of discontinuous functions}
We start by showing the effectiveness of the method on four discontinuous functions depicted in  Figure~\ref{fig:argmin4} alongside their polynomial argmin approximations. The functions are
\[
f_1 = \begin{cases}
-1,& x \in [-0.75,0.75] \\
1,& x\in [-1,-0.75)\cup (0.75,1],
\end{cases}\;\;
f_2 = \begin{cases}
-1,& x \in [-0.75,-0.25] \cup [0.25,0.75]  \\
1,& x\in [-1,-0.75)\cup (-0.25,0.25) \cup (0.75,1].
\end{cases}
\]
\[
f_3 = x^2f_1,\quad f_4 = \sin(2x) f_1.
\]
 In each case we used 200 data points sampled uniformly at random in $[-1,1]$ and solved~(\ref{opt:sos}). We observe a remarkably precise recovery of the discontinuous functions.
 In (\ref{eq:hkparam}), the minimal degrees $d_x$ and $d_y$ required to obtain an approximation of this accuracy are reported in the figure. For comparison, in Figure~\ref{fig:argminvsNN} we depict the performance on the same task with a neural network with five hidden layers with 20 neurons per layer with the hyperbolic tangent activation functions, trained using Matlab's neural network toolbox with gradient descent; these parameters were selected by manual hyperparameter tuning.

 \subsection{Univariate: Parameter dependence}
 Here we investigate the dependence of the approximation quality on $d_x$ which is the degree of the polynomials $h_k$ parametrizing the polynomial $p$ in (\ref{eq:hkparam}). We do so on the function from Eq. (66) in \cite{eckhoff} that possesses 7 discontinuities. For data, we use two hundred equidistantly spaced samples in the interval $[-1,1]$. The results are depicted in Figure~\ref{fig:argmin7discont}. As expected the approximation quality improves as the degree of $h_k$ increases, obtaining a very precise accuracy for $\mathrm{deg}\,h_k = 7$. For comparison, Figure~\ref{fig:argmin7discont_NN} depicts also the results with a neural network approximation. 

\begin{figure*}[ht]
\begin{picture}(140,350)
\put(-150,170){\includegraphics[width=72mm]{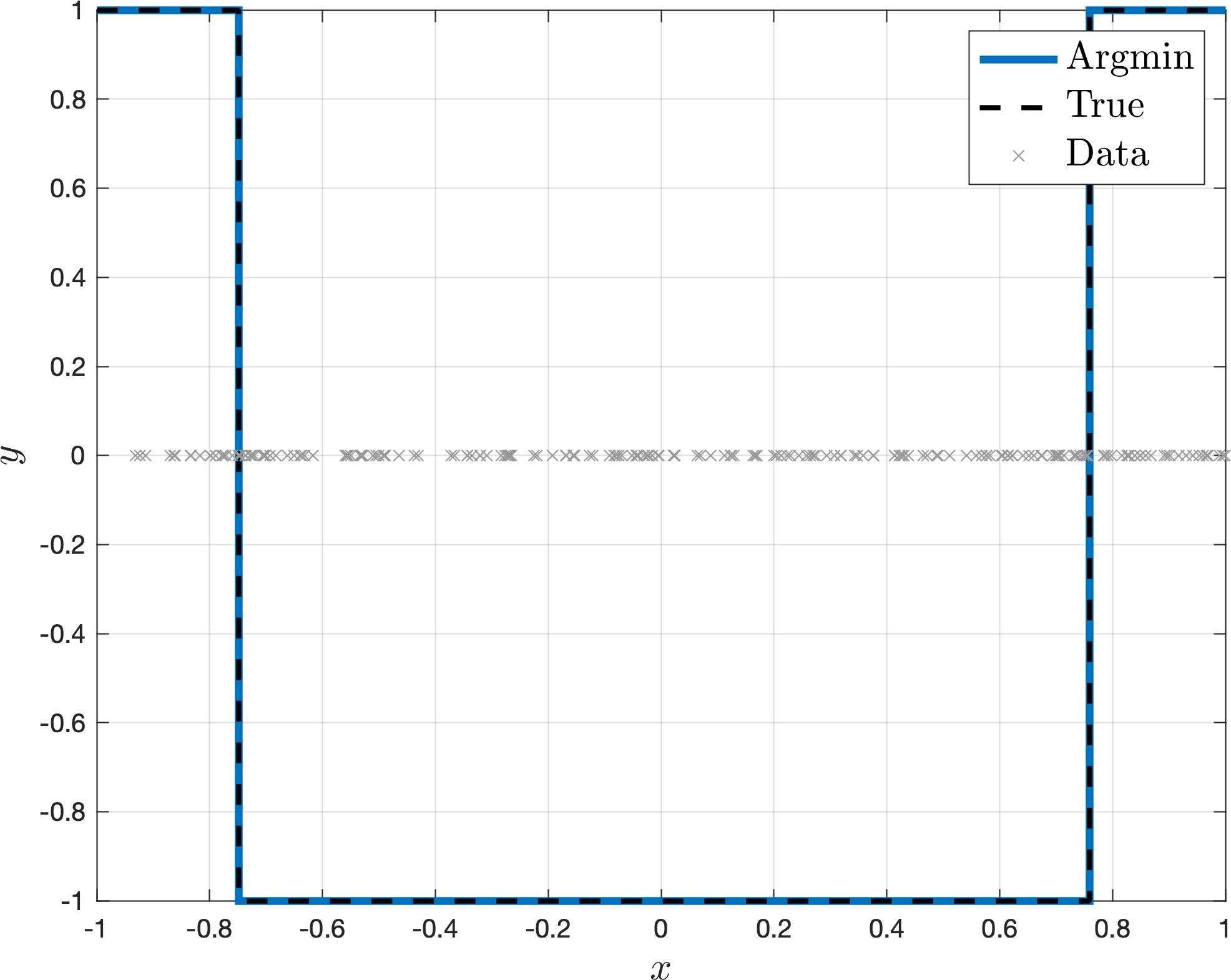}}
\put(100,170){\includegraphics[width=72mm]{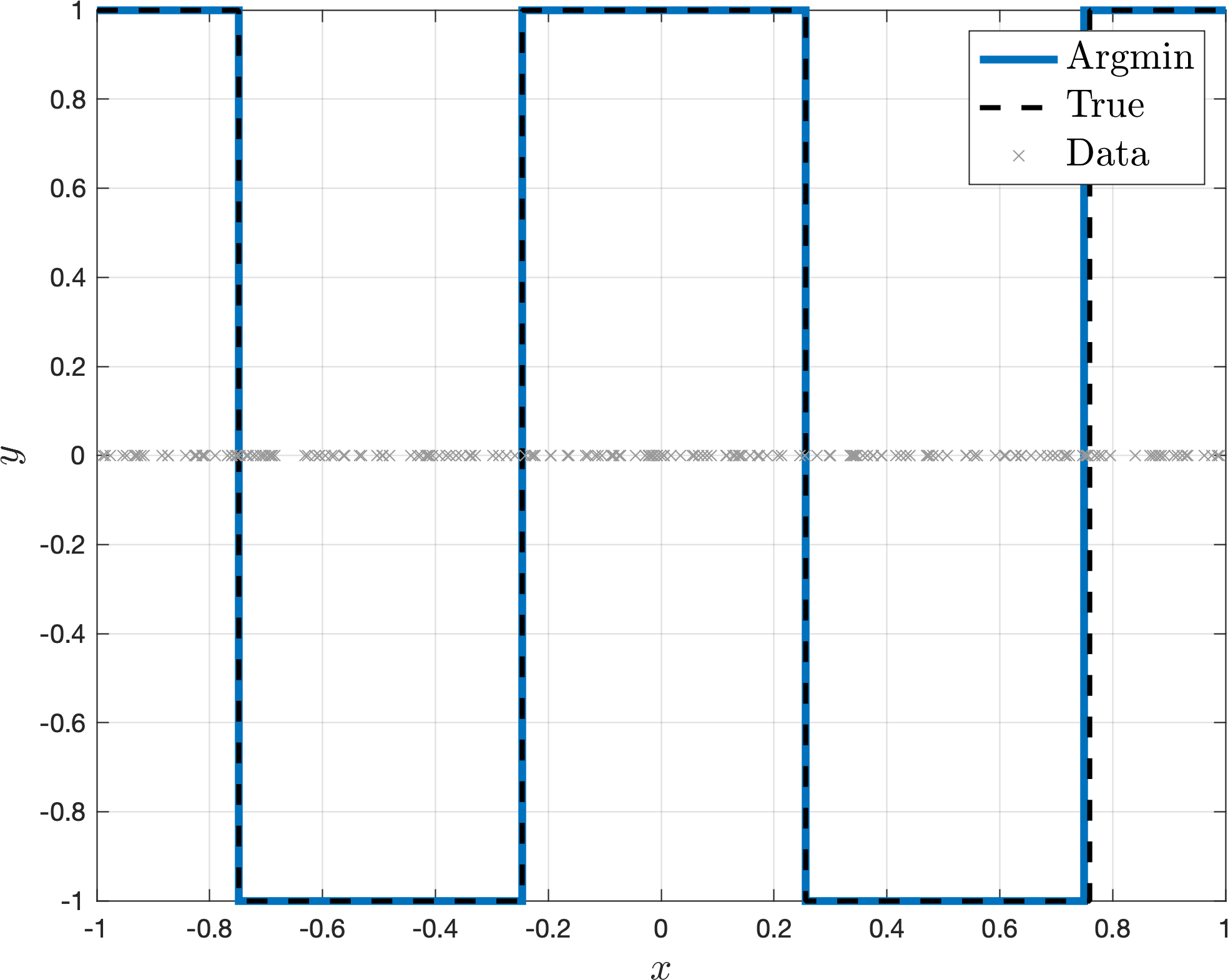}}

\put(-150,-5){\includegraphics[width=72mm]{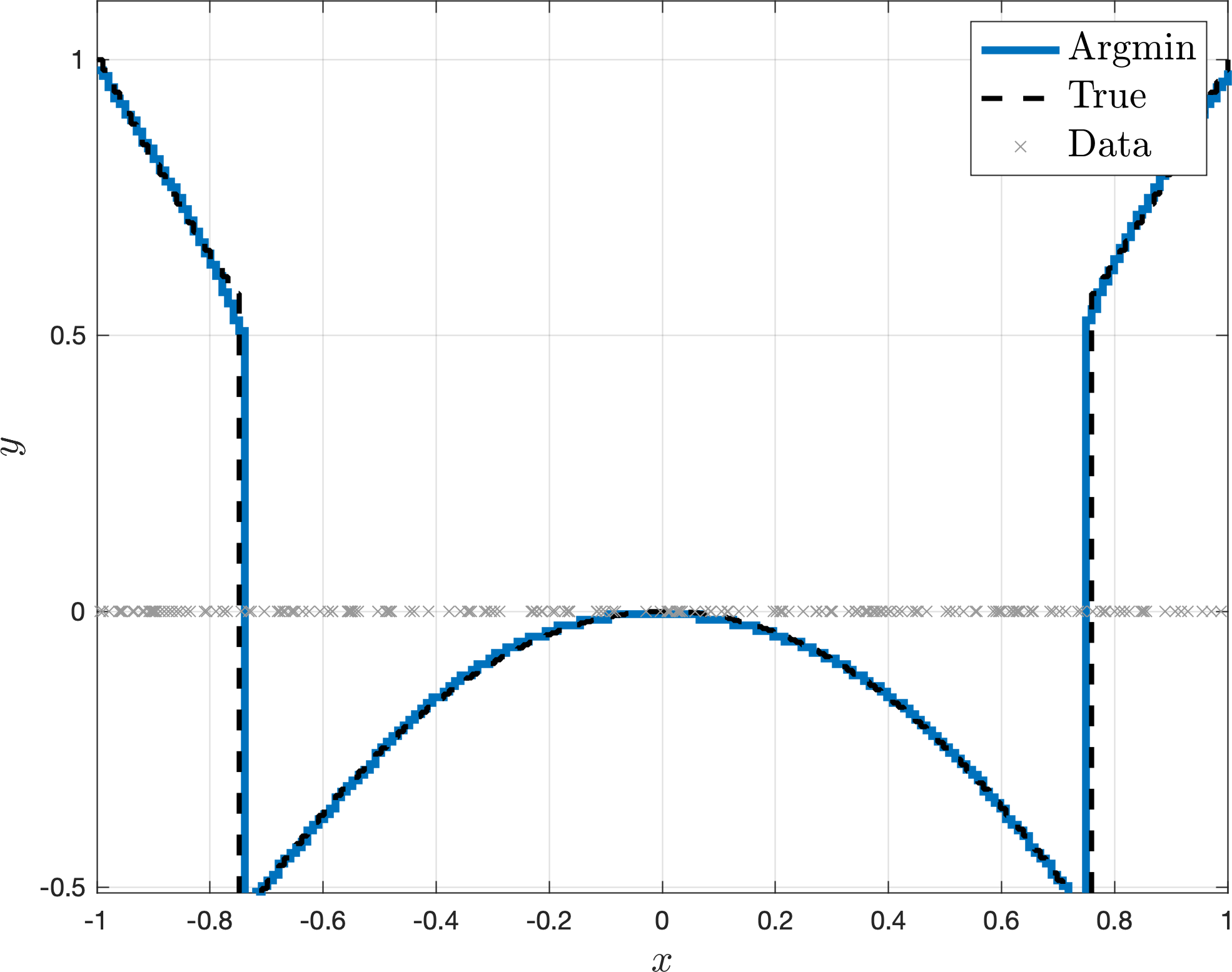}}
\put(100,-5){\includegraphics[width=72mm]{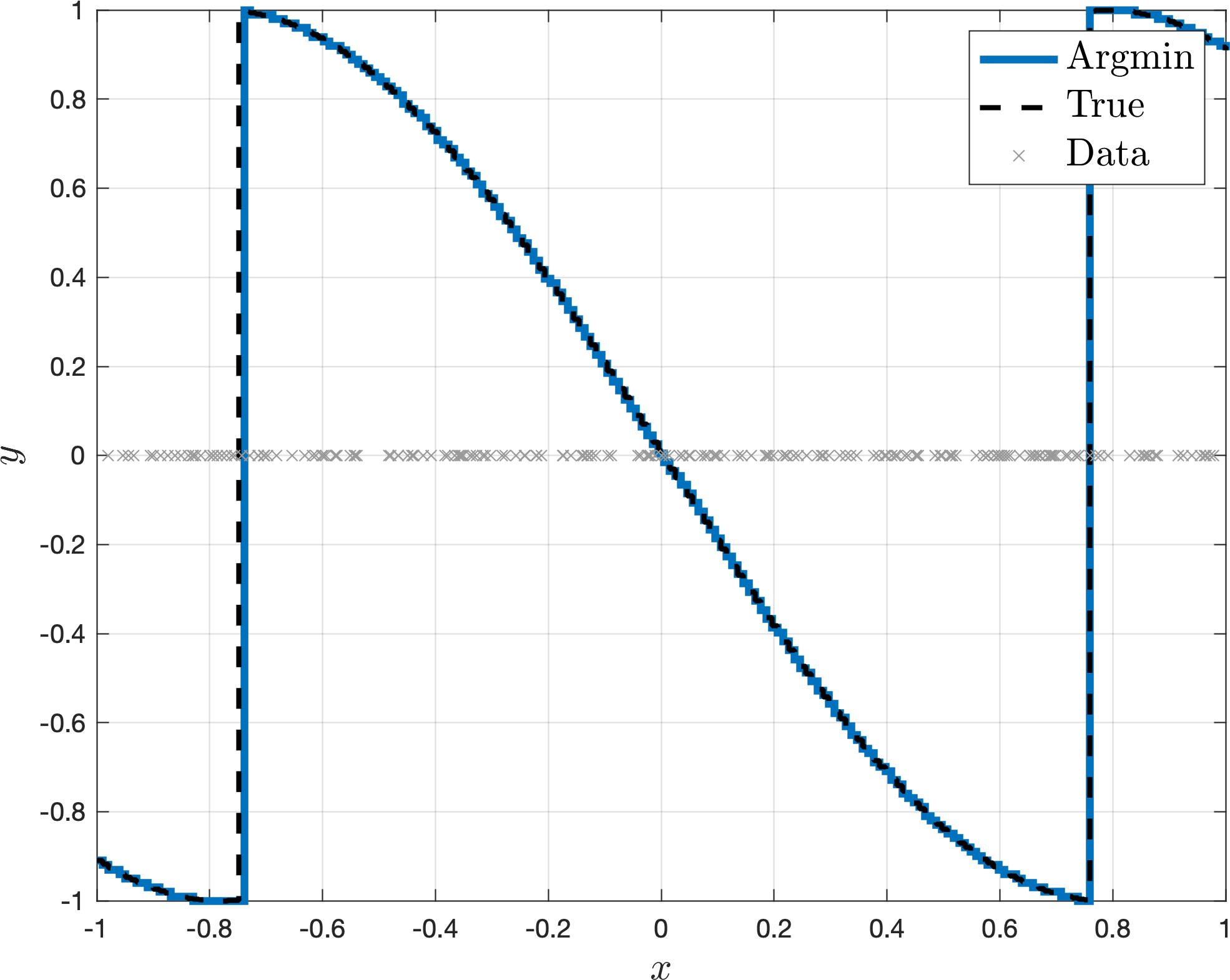}}

\put(-60,210){\scriptsize $d_y = 1$}
\put(-60,200){\scriptsize $d_x= 2$}
\put(-45,315){$f_1$}
\put(207,315){$f_2$}
\put(-45,140){$f_3$}
\put(207,140){$f_4$}

\put(192,210){\scriptsize $d_y = 4$}
\put(192,200){\scriptsize $d_x= 4$}

\put(-60,30){\scriptsize $d_y = 4$}
\put(-60,20){\scriptsize $d_x= 4$}

\put(192,30){\scriptsize $d_y = 4$}
\put(192,20){\scriptsize $d_x = 4$}

\end{picture}
\caption{\small  Polynomial argmin approximations of different functions with discontinuities.}
\label{fig:argmin4}
\end{figure*}
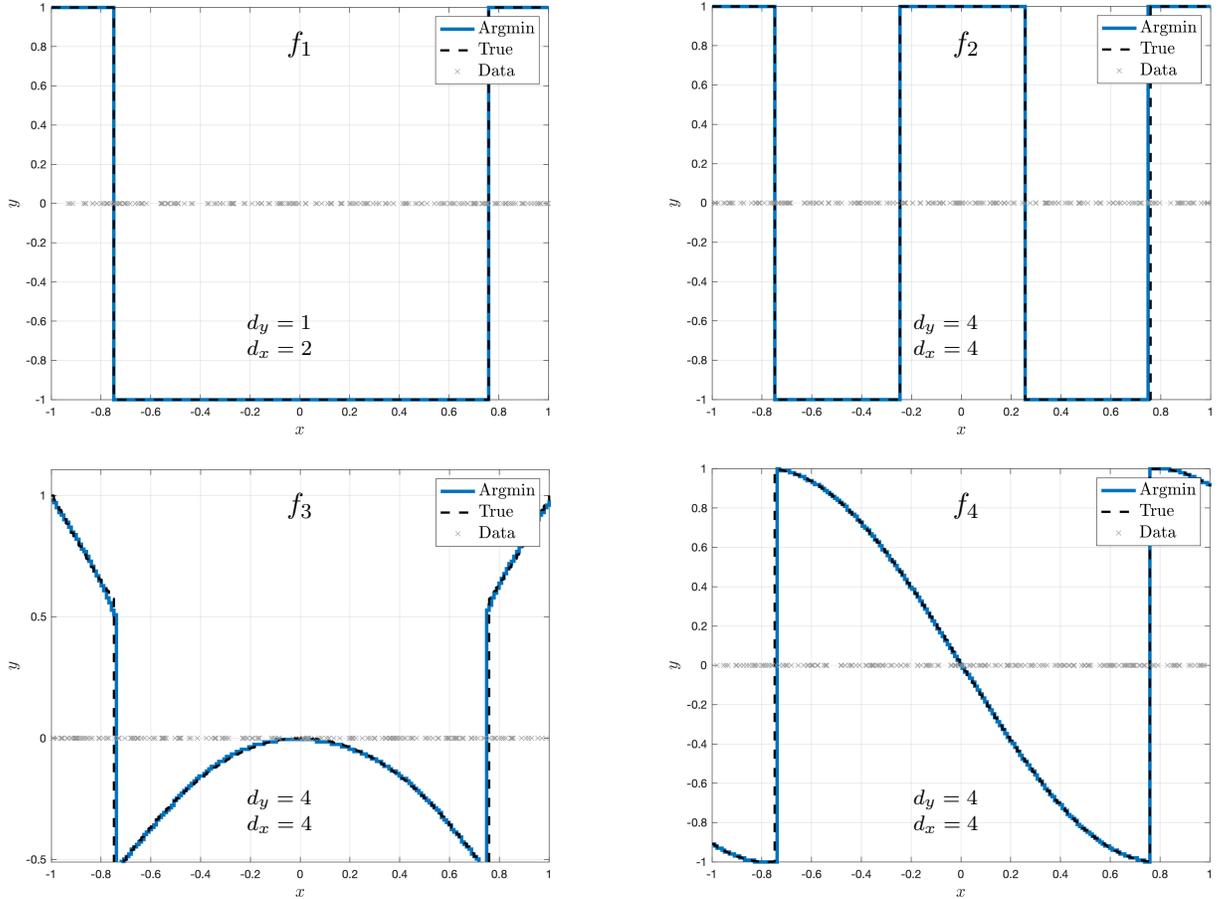

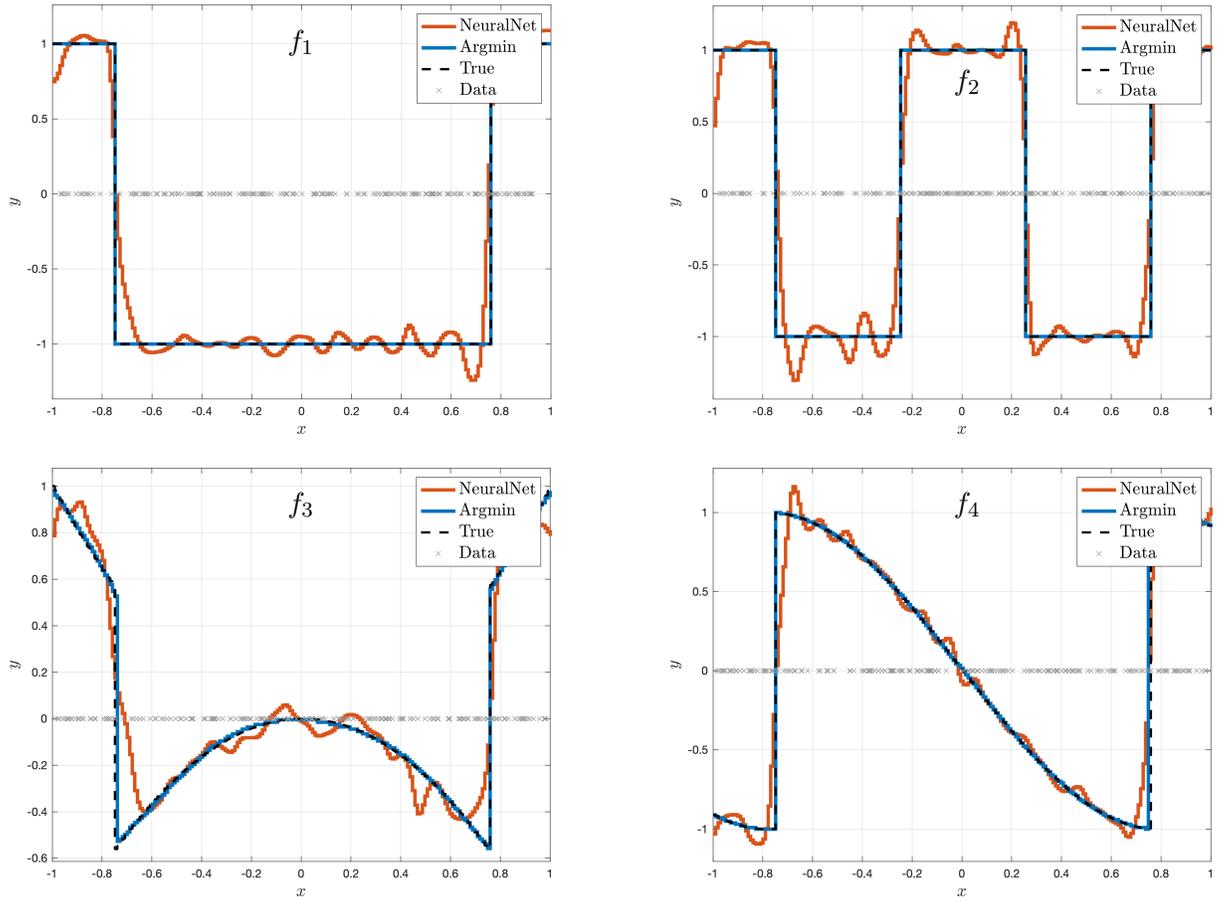
\begin{figure*}[ht]
\begin{picture}(140,350)
\put(-150,170){\includegraphics[width=72mm]{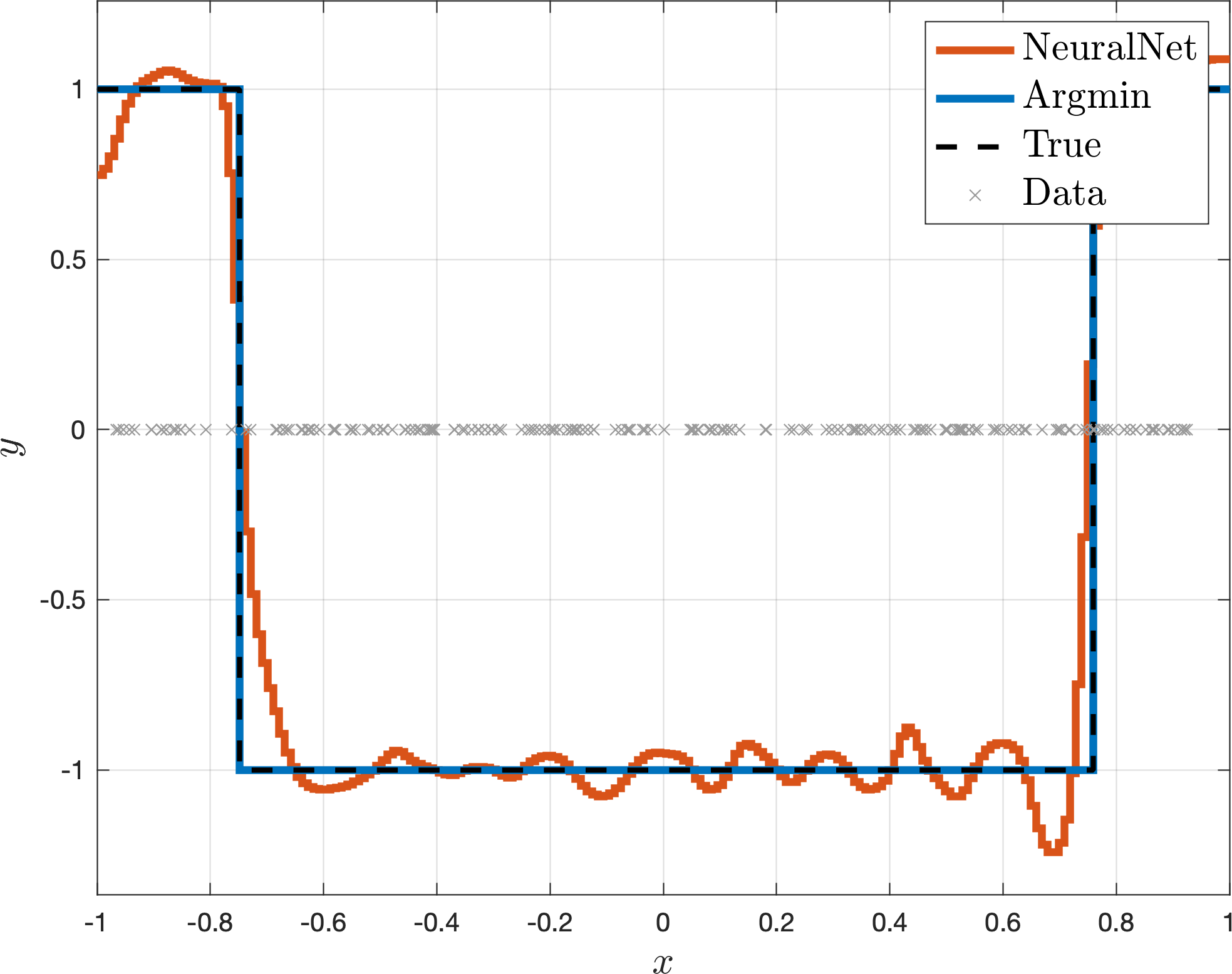}}
\put(100,170){\includegraphics[width=72mm]{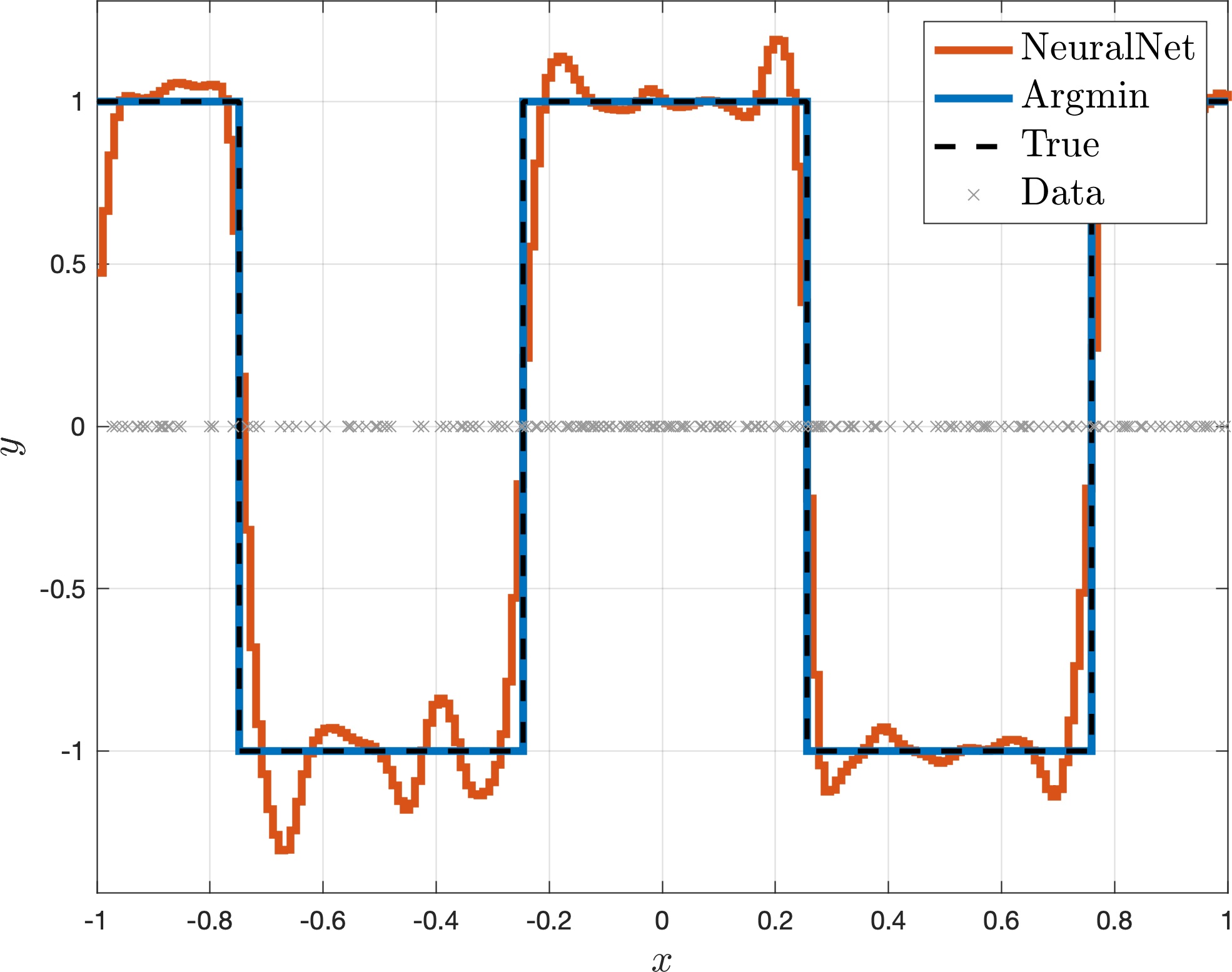}}

\put(-150,-5){\includegraphics[width=72mm]{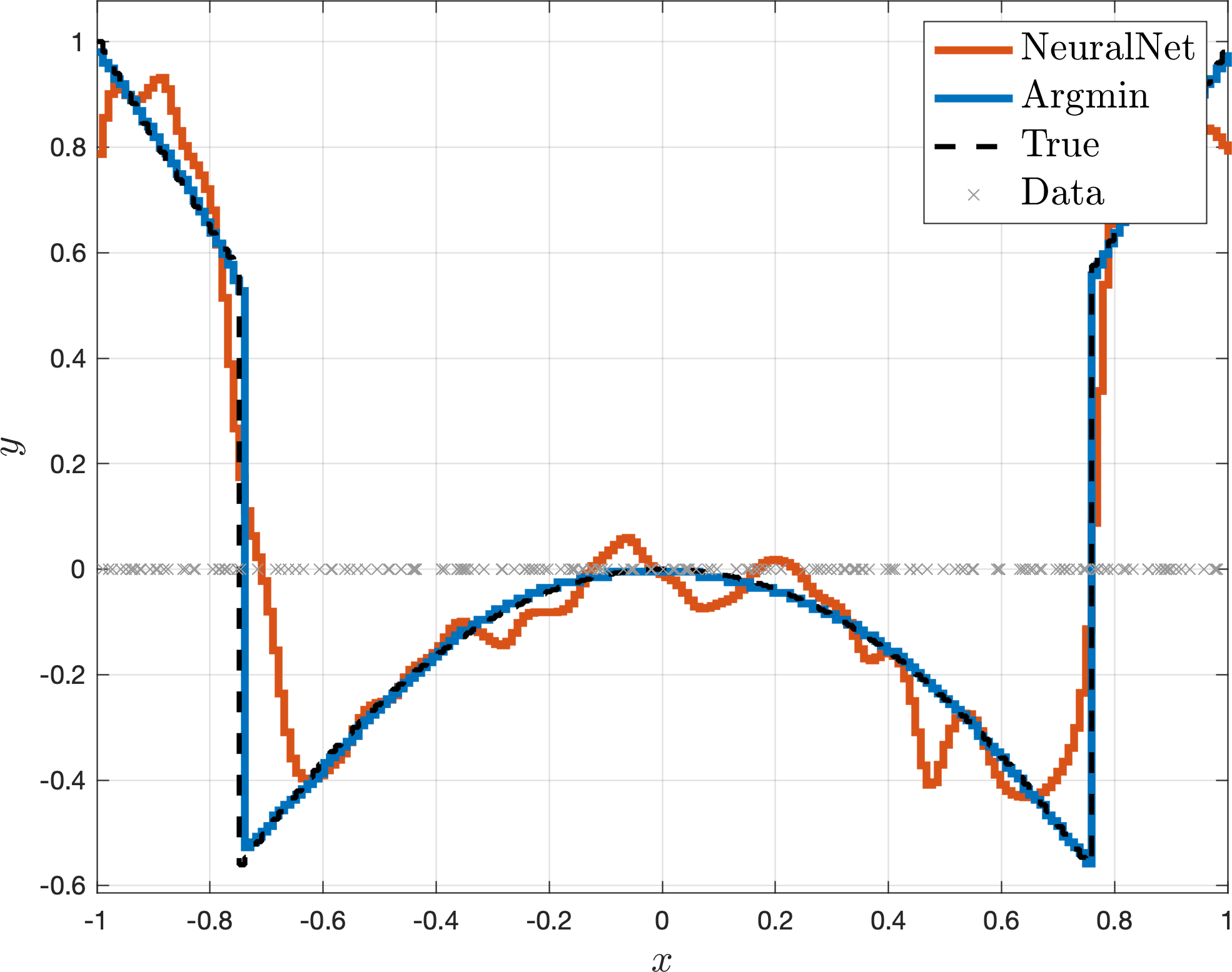}}
\put(100,-5){\includegraphics[width=72mm]{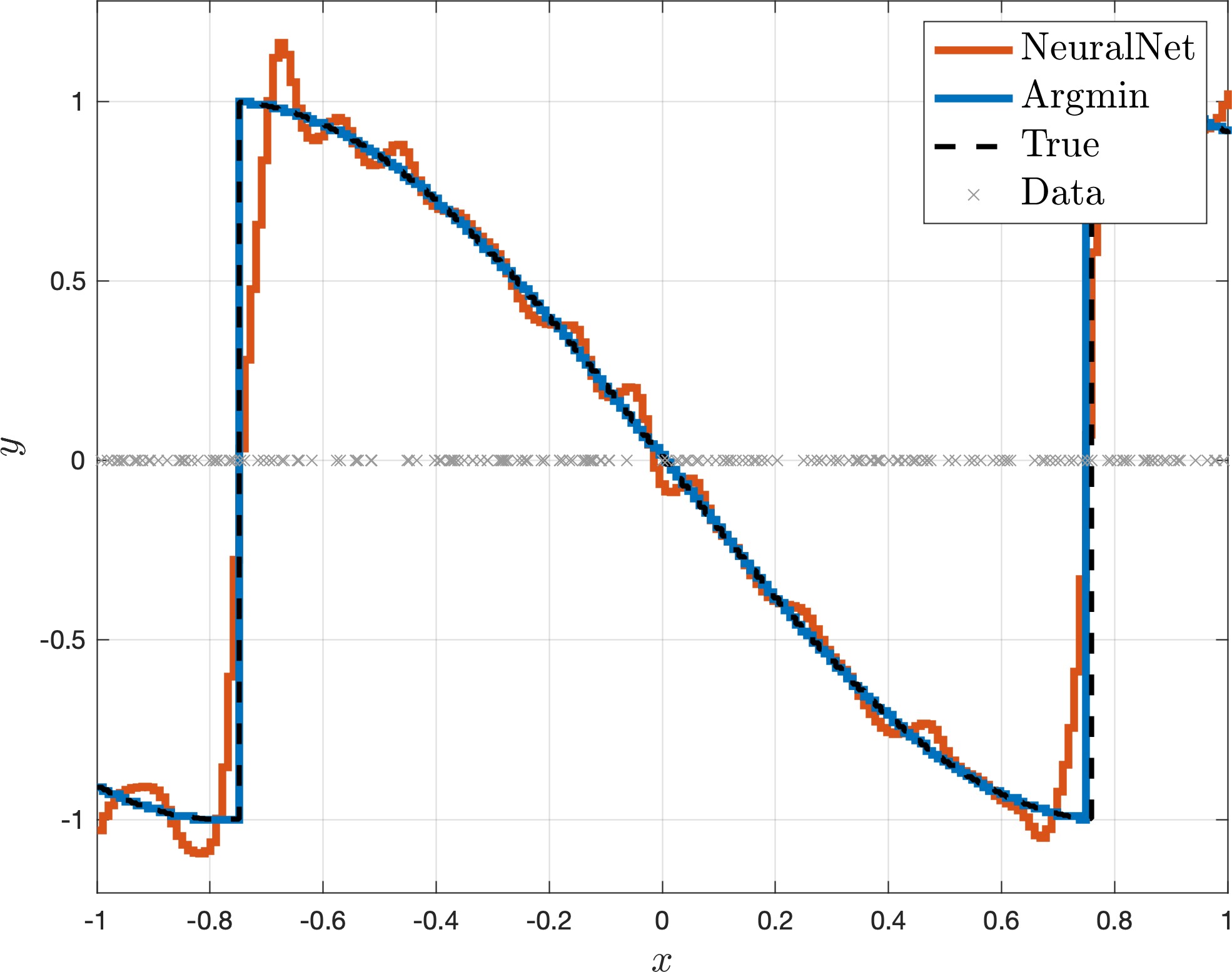}}

\put(-45,315){$f_1$}
\put(207,300){$f_2$}
\put(-45,140){$f_3$}
\put(207,140){$f_4$}

\end{picture}
\caption{\small  Polynomial argmin approximation versus neural network.}
\label{fig:argminvsNN}
\end{figure*}

\begin{figure*}[ht]
\begin{picture}(140,350)
\put(-150,170){\includegraphics[width=72mm]{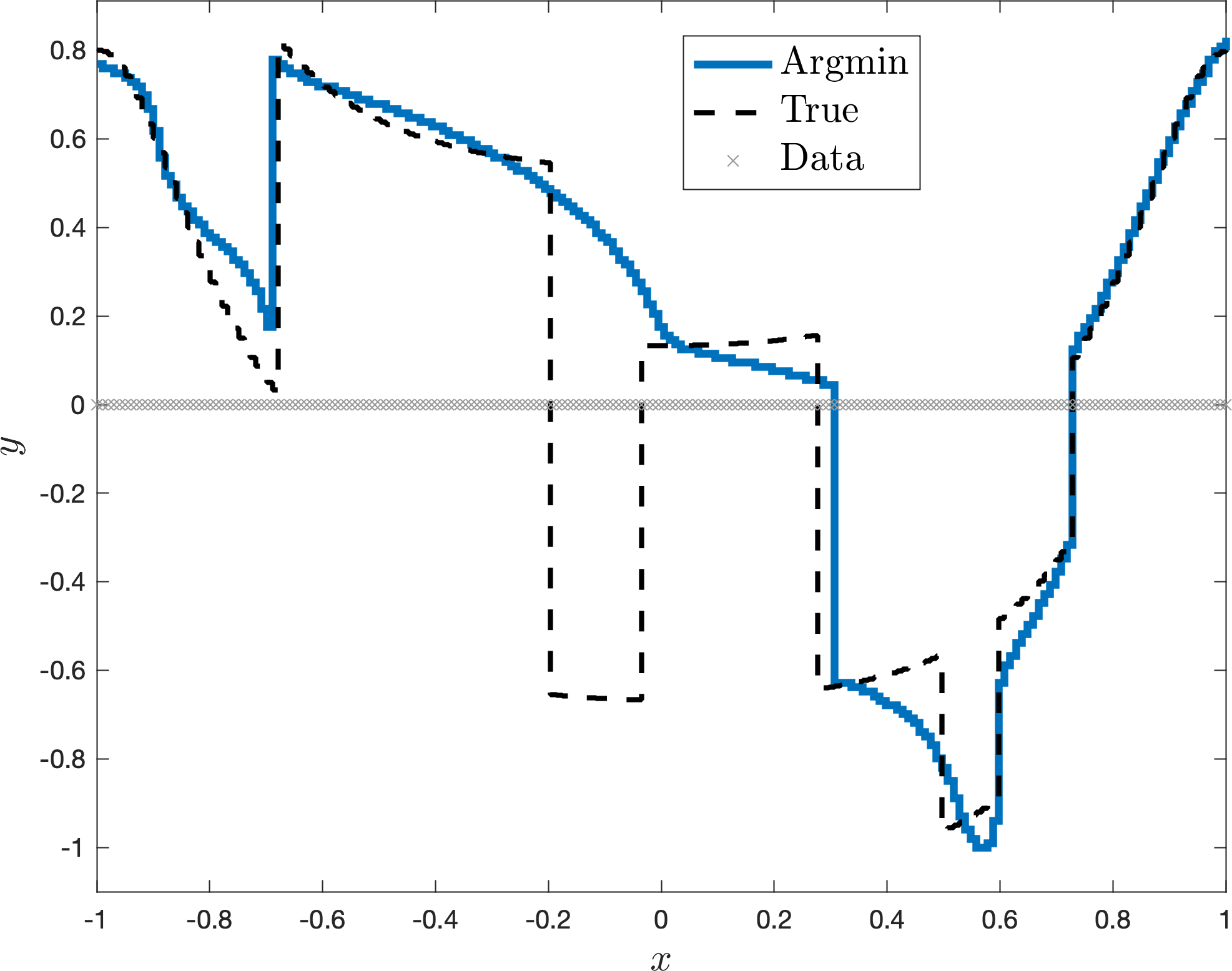}}
\put(100,170){\includegraphics[width=72mm]{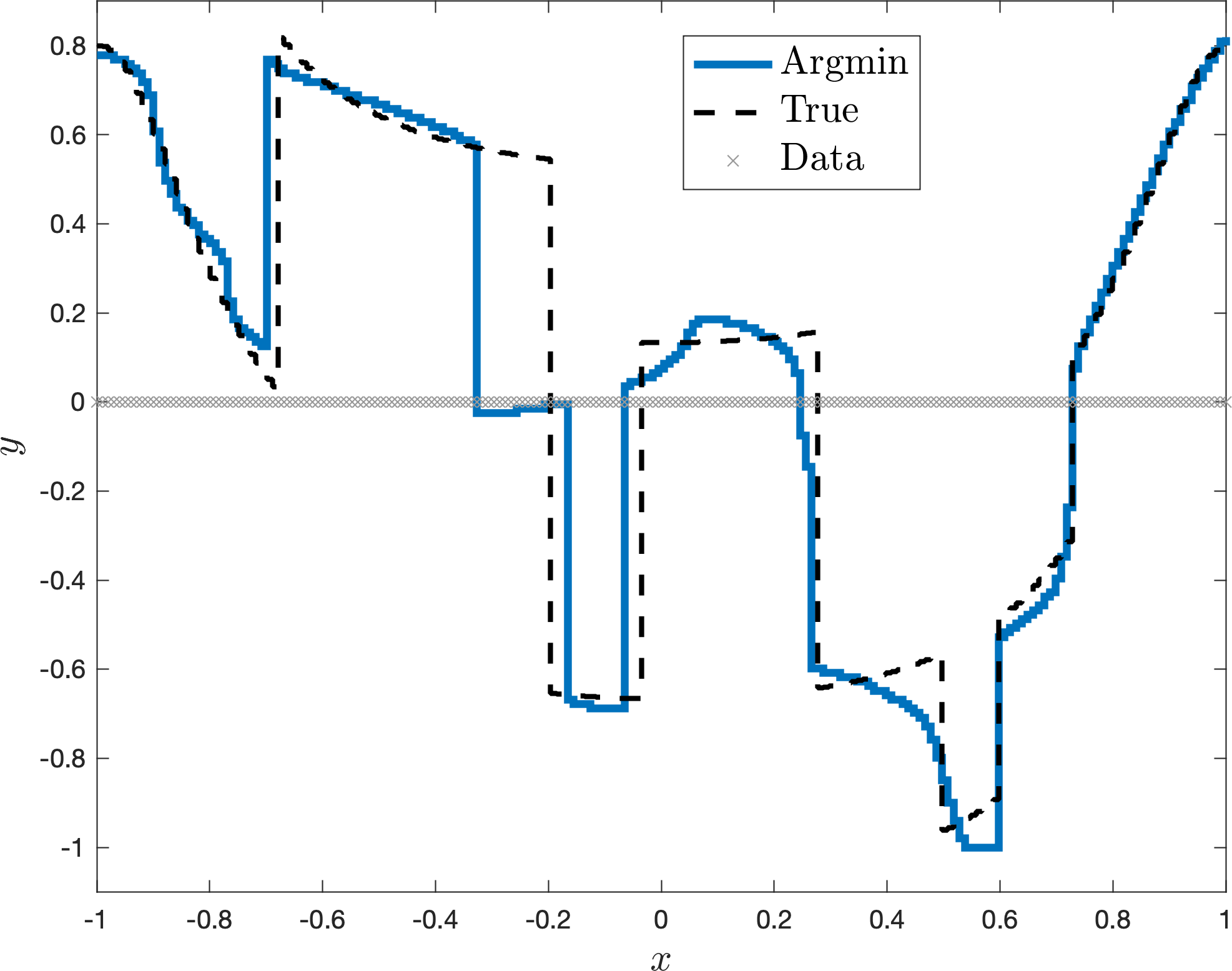}}
\put(-150,-5){\includegraphics[width=72mm]{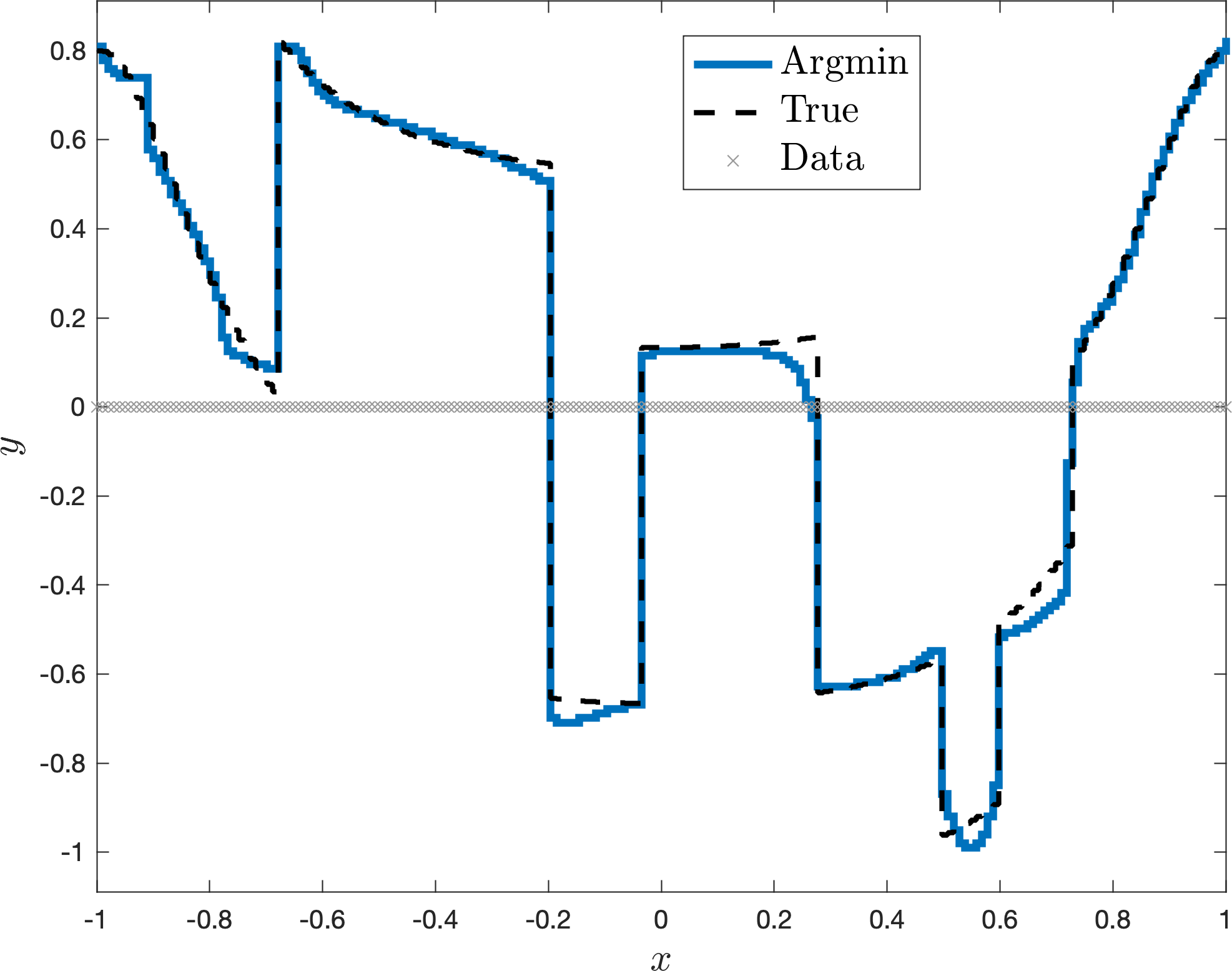}}
\put(100,-5){\includegraphics[width=72mm]{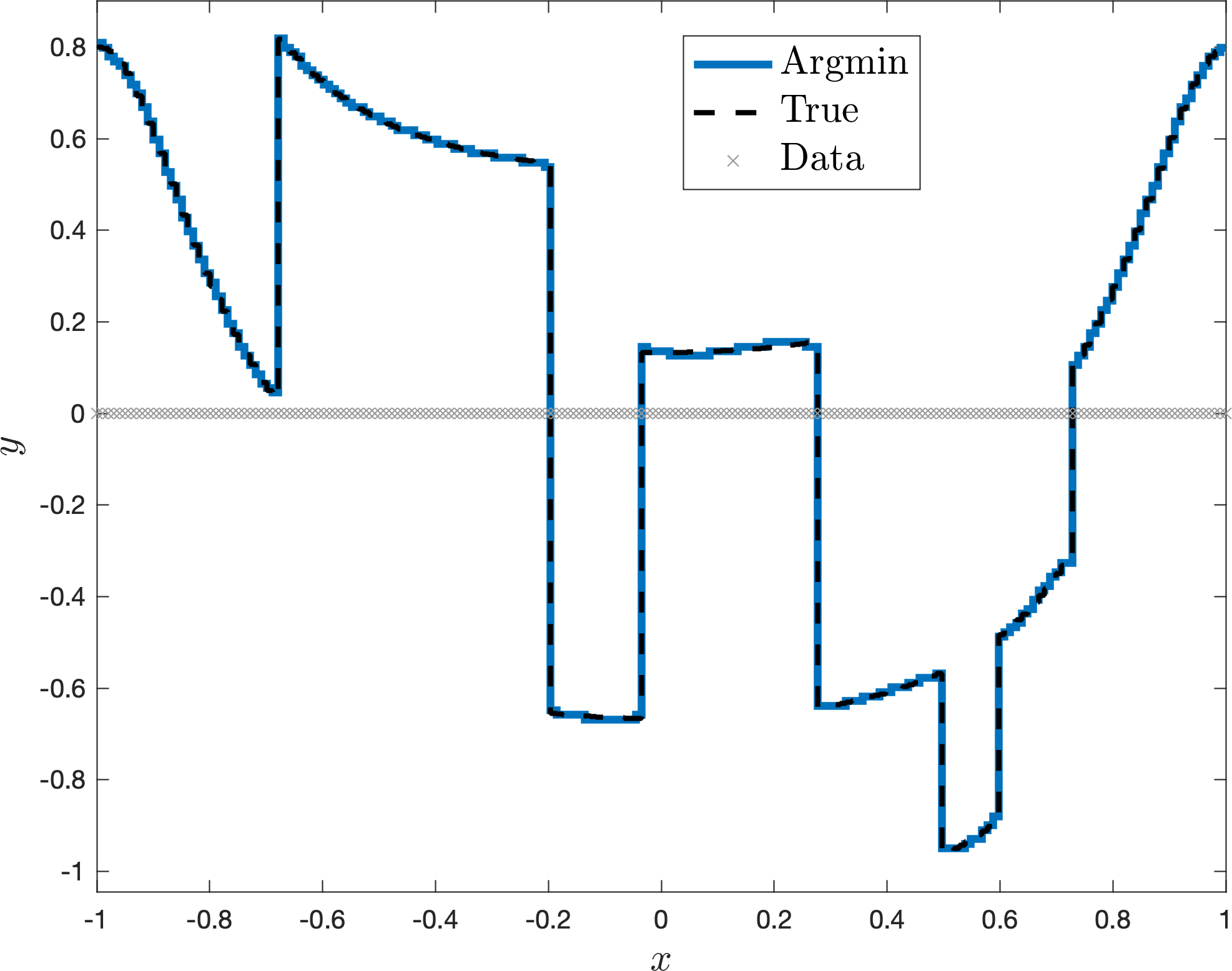}}

\put(-120,200){\small $d_x = 4$}
\put(130,200){\small $d_x = 5$}
\put(-120,25){\small $d_x = 6$}
\put(130,25){\small $d_x = 7$}

\end{picture}
\caption{\small  Polynomial argmin approximation on a function with 7 discontinuities with $d_y = 6$ for different values of $d_x$.}
\label{fig:argmin7discont}
\end{figure*}
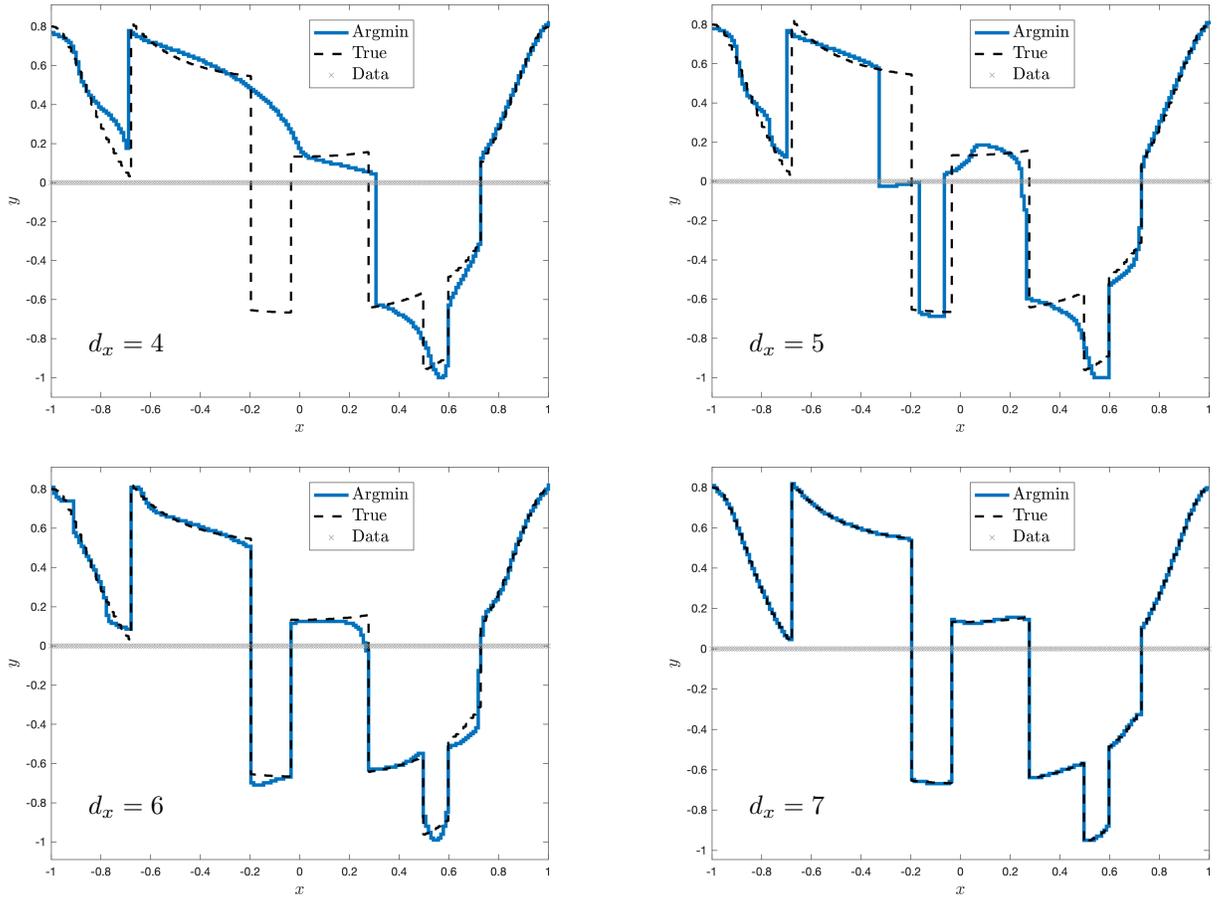

\begin{figure*}[th]
\begin{picture}(140,350)
\put(-150,170){\includegraphics[width=72mm]{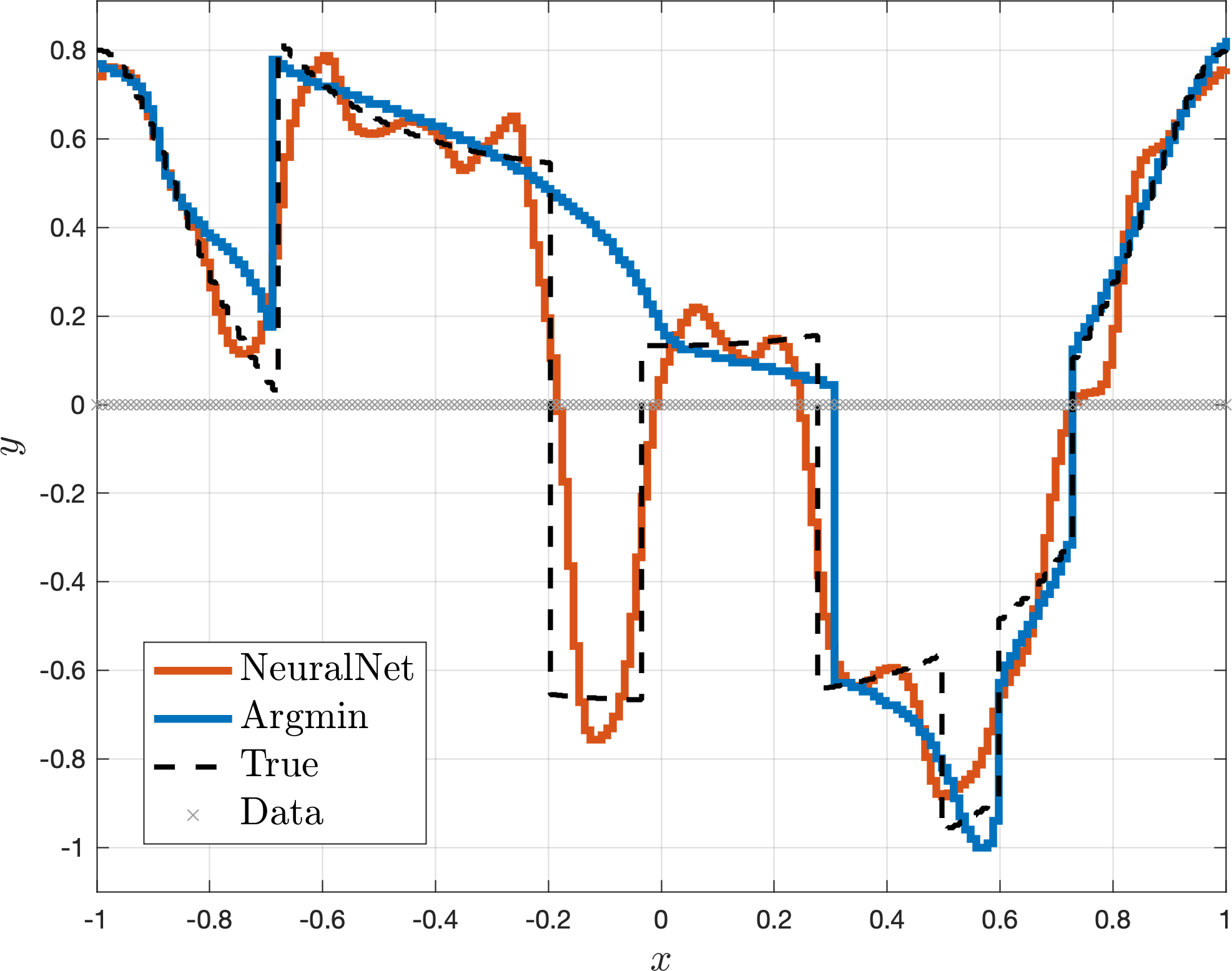}}
\put(100,170){\includegraphics[width=72mm]{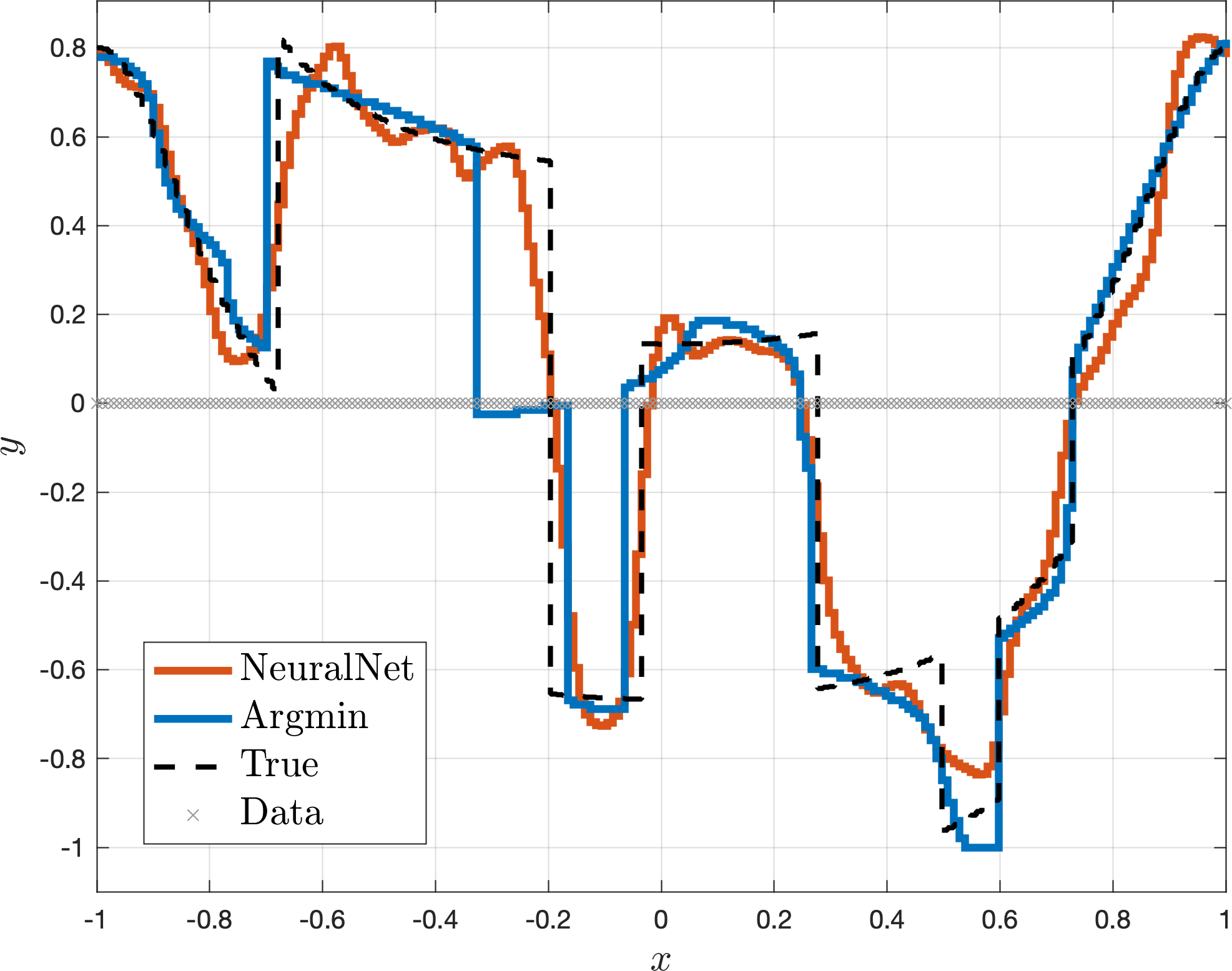}}
\put(-150,-5){\includegraphics[width=72mm]{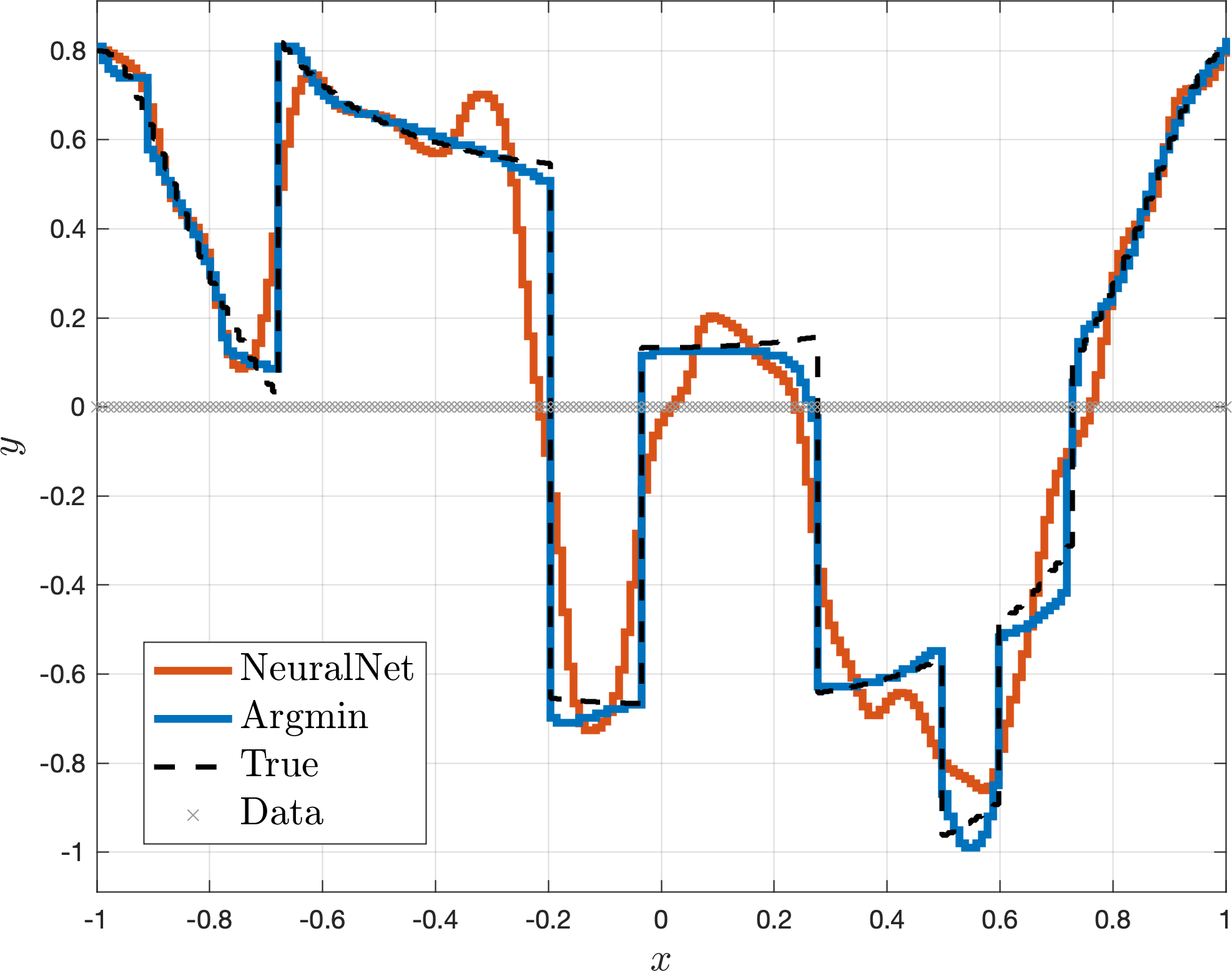}}
\put(100,-5){\includegraphics[width=72mm]{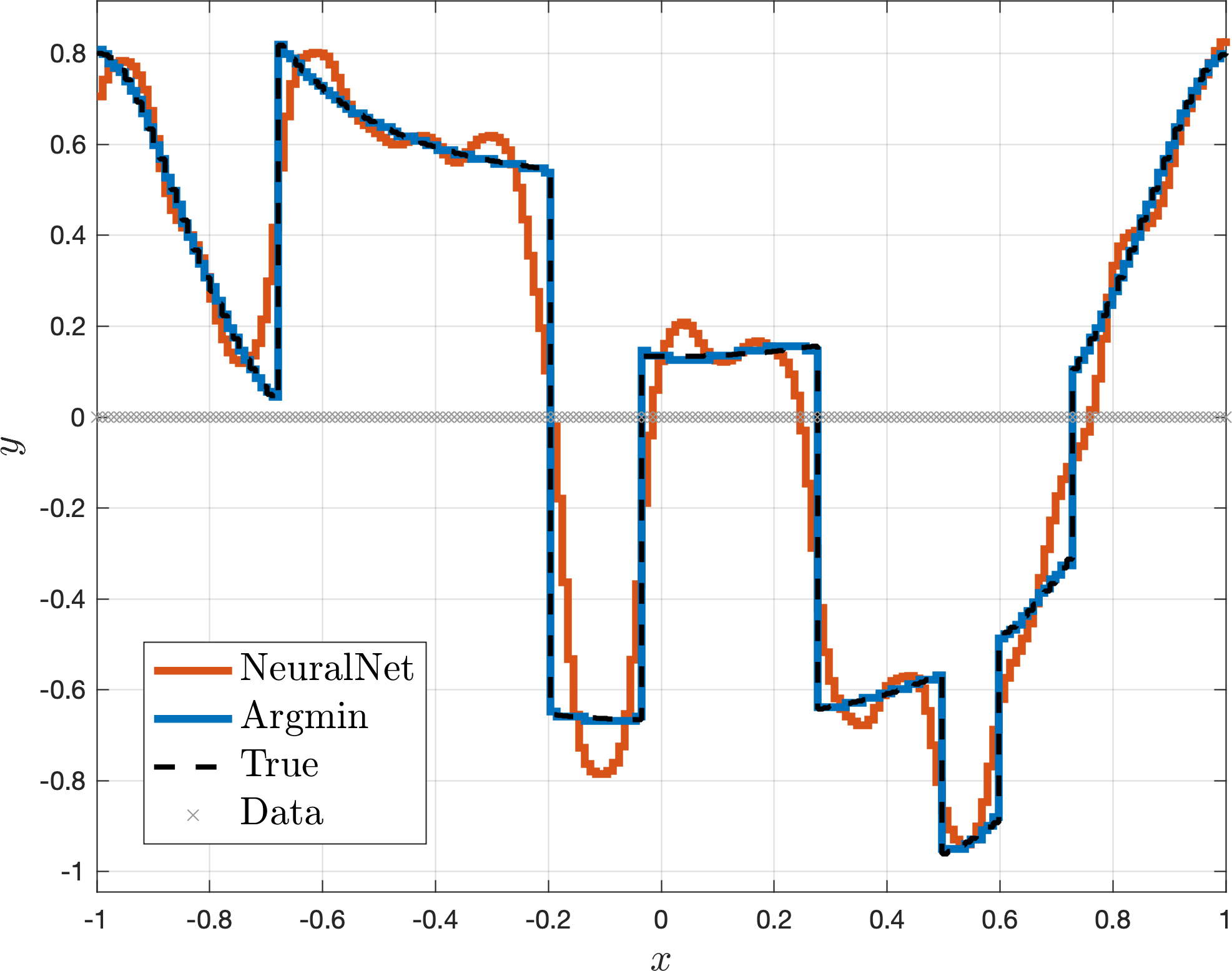}}

\put(-125,230){\small $d_x = 4$}
\put(125,230){\small $d_x = 5$}
\put(-125,55){\small $d_x = 6$}
\put(125,55){\small $d_x = 7$}

\end{picture}
\caption{\small  Polynomial argmin approximation of a function with 7 discontinuities for different values of the degree of $h_k$, $k=1,\ldots,6$ in comparison with a neural network. }
\label{fig:argmin7discont_NN}
\end{figure*}

 \subsection{Univariate: Challenging continuous functions}
 For completeness we briefly report results for approximation of continuous functions. We do so on two functions. The first one is $f(x) = \sqrt{|\sin\,x|}$ which is a transcendental function with H\"older exponent 1/2 whose derivative grows unbounded near the origin. The second one is the Runge function $f(x) = (1 + 25x^2)^{-1}$ which is a smooth function that exhibits the  Runge phenomenon (oscillations near the boundary) when approximated by polynomials through interpolation. The results are depicted in Figure~\ref{fig:continuous}. As in the previous examples, we observe an accurate fit and no oscillations with low degrees of $p$ in $x$ and $y$.
 
 \begin{figure*}[ht]
\begin{picture}(140,170)

\put(-150,-5){\includegraphics[width=72mm]{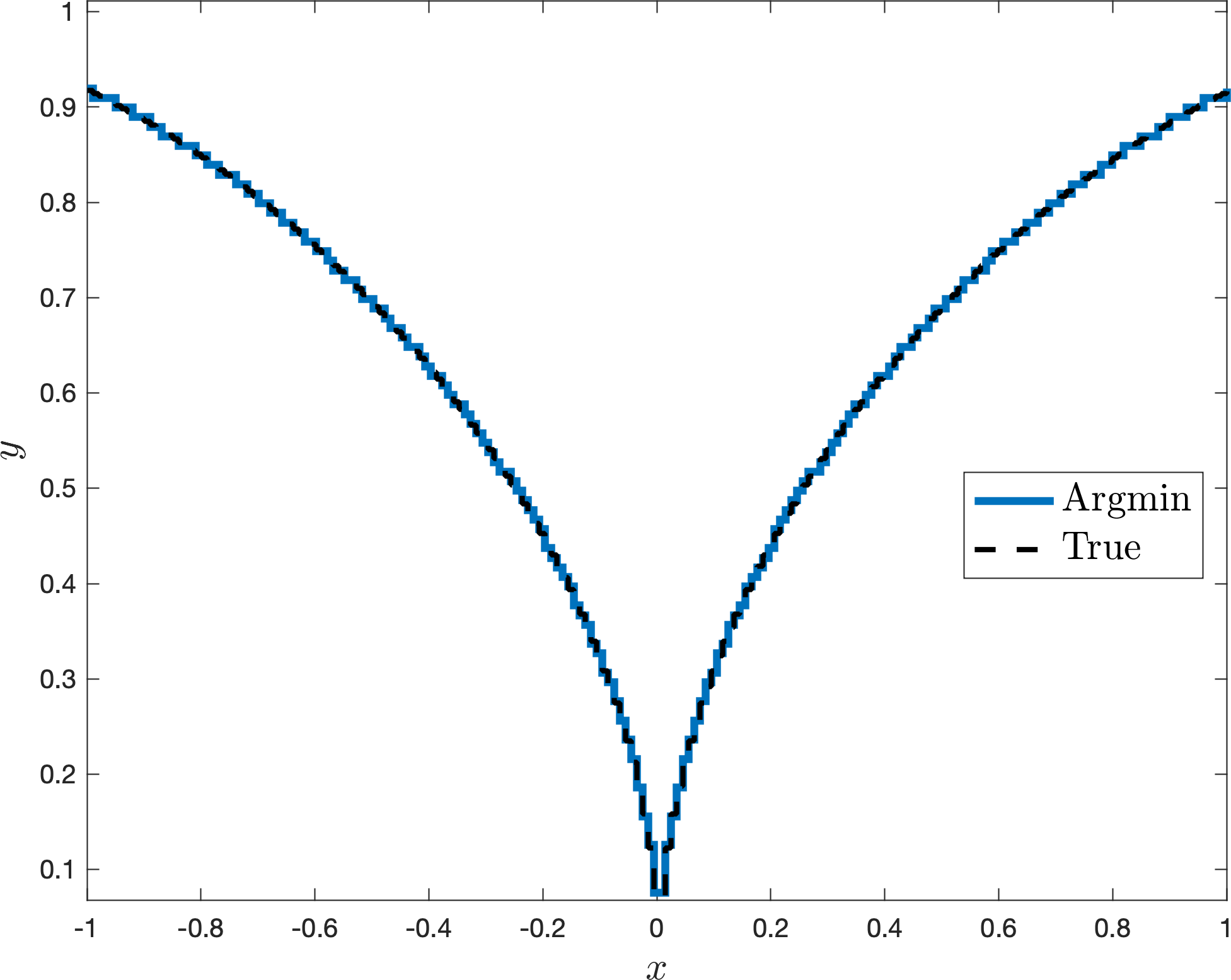}}
\put(100,-5){\includegraphics[width=72mm]{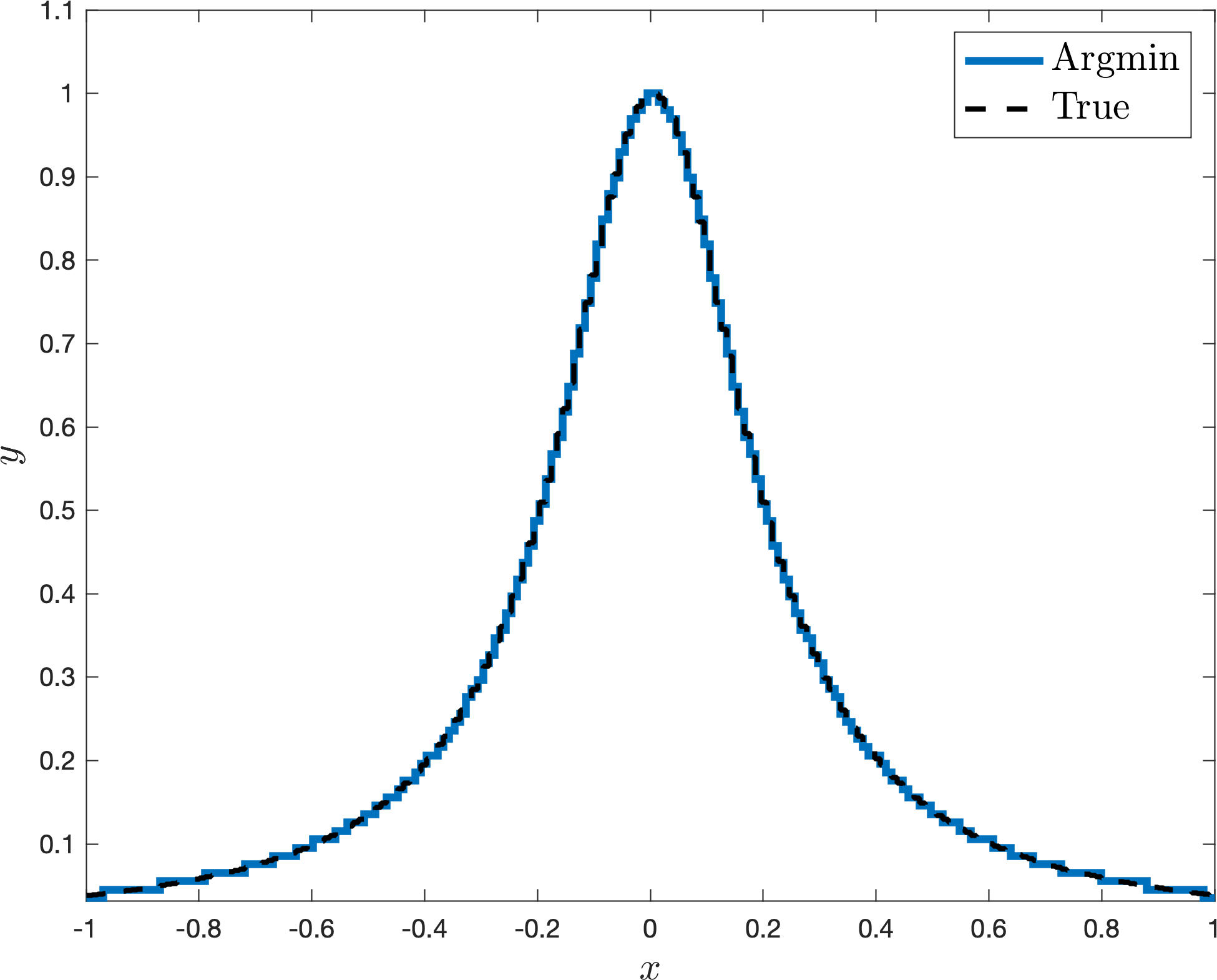}}

\put(-60,140){ $\sqrt{|\sin x|}$}
\put(186,25){\footnotesize $(1 + 25x^2)^{-1}$}
\end{picture}
\caption{\small  Polynomial argmin approximations of two continuous functions with $d_x = 4$, $d_y = 4$. Left: transcendental function with unbounded derivatives near the origin. Right: Runge function.}
\label{fig:continuous}
\end{figure*}

\subsection{Bivariate: Approximation of discontinuous functions} ~

{Consider the discontinuous function
\[
f(\bx) = \left\{\begin{array}{lll}
\frac{1+x_1+x_2}{2} & \mathrm{if} & x^2_1+x^2_2 \leq \frac{1}{4} \\
0 & \mathrm{otherwise}
\end{array}\right.
\]
constructed by multiplying an affine function with the
indicator function\footnote{The indicator function of a set is equal to one on the set and zero outside.} of a bivariate disk.
On Figure \ref{fig:lindiskchebfun2} we represent the {\tt chebfun2} approximation obtained with the {\tt chebfun} package \cite{chebfun}:\\
{\tt [x1,x2]=meshgrid(linspace(-1,1,100));}\\
{\tt plot(chebfun2(double(x1.^2+x2.^2<=1/4).*(x1+x2+1)/2));}\\
We observe that the approximation is corrupted by the typical Gibbs phenomenon encountered when approximating a discontinuous function with polynomials \cite[Chapter 9]{atap}, namely large oscillations near the discontinuity set.
\begin{figure}[ht]
	\centering
	\includegraphics[width=0.6\textwidth]{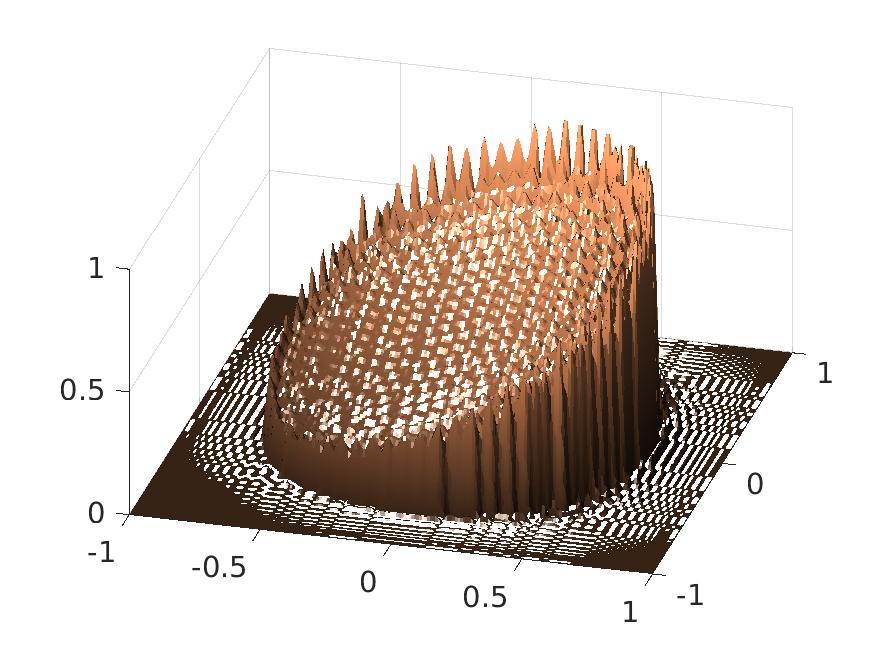}
	\caption{Chebyshev polynomial approximation of a discontinuous bivariate function obtained by {\tt chebfun2}.
		\label{fig:lindiskchebfun2}}
\end{figure}
\begin{figure*}[ht]
	\centering
	\includegraphics[width=0.48\textwidth]{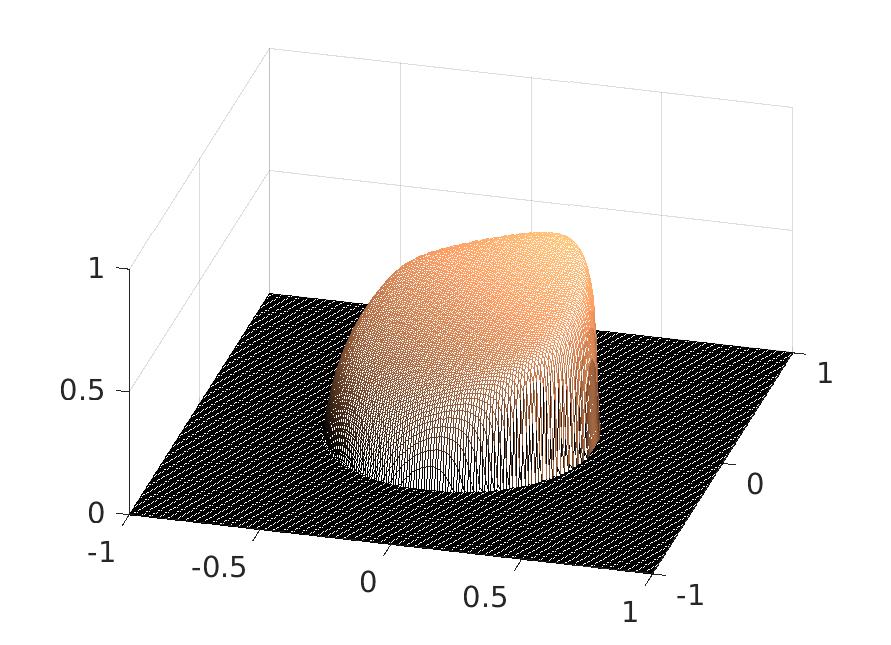}
    \includegraphics[width=0.48\textwidth]{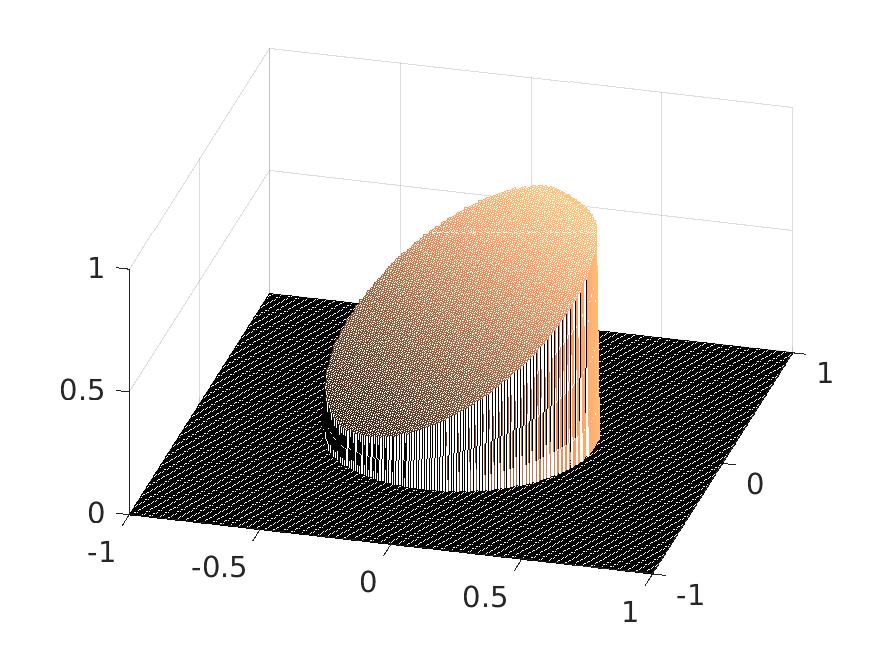}	
	\caption{Polynomial argmin approximation of a discontinous function for $d_x=3, d_y=2$ (left) and $d_x=3, d_y=4$ (right). \label{fig:lindisk}}
\end{figure*}
}

\section{Discussion and conclusion}
We have presented a simple method based on the argmin of a polynomial for approximation of discontinuous functions. The approach is model-free and mesh-free {in the sense that it
does not require prior knowledge about the function being approximated as it works only with samples of its values. 
It is grounded in powerful tools from \emph{univariate} sum of squares optimization,  
hence based only on a very specific class of convex semidefinite programming, and so it is simple to use.
It shows a great promise in numerical examples and we believe that it can become a valuable tool in data analysis. We have also proved that \emph{exact} recovery is possible on certain examples of discontinuous functions and have provided theoretical analysis of in-sample and out-of-sample error in a probabilistic setting. 
In the argmin approach \cite{constructive-1} based on the Christoffel-Darboux polynomial, such an exact recovery
is not possible in general as an $\varepsilon$-regularization term is introduced to guarantee that the associated moment matrix is non-singular. In addition, the size 
of the moment matrix to invert strongly depends on the dimension of data while
in our optimization-based appproach, the size and the number of resulting matrices to be positive semidefinite does \emph{not} depend on the dimension of the data
(their number is linear in the sample size). }


{While we  have used general purpose semidefinite solvers to construct our argmin approximants, more efficient approaches can be envisioned. Indeed our formulation boils down to optimization over the cone of univariate non-negative polynomials, a very specific class of semidefinite optimization problems. For example, non-symmetric solvers may perform faster on these problems \cite{py19}. Another option could be to bypass numerical optimization and use tailored numerical linear algebra as in \cite{f16}.
}

{
An additional interesting feature of the approximant is that its evaluation at a given point $\x\in\X$ reduces
to finding the global minimum of a univariate polynomial on an interval, which can be done efficiently e.g. by matrix eigenvalue computation. A numerically stable algorithm is described in \cite[Section 7]{battles} and implemented in the {\tt roots} function of the {\tt chebfun} package \cite{chebfun}. It is based on the application of the QR algorithm for finding the eigenvalues of a balanced companion matrix constructed by evaluating the polynomial at Chebyshev points.
}

This paper is a first step that introduces the argmin approximant
and illustrates its promising potential on non trivial numerical examples. We hope that
it could inspire some further developments. 
In particular, we have left open 
the question of optimal rates of convergence of the argmin approximant 
or more generally its worst-case performance when considering pre-defined classes of functions to approximate, 
e.g. in terms of the {\emph{manifold width} discussed in  \cite{cohen}, which is a generalization of the classical Kolmogorov width}.
Based on Section~\ref{sec:polynomials}, it is clear that the rates are at least as good as those of polynomial approximation whenever the degree of $p$ in $y$ is at least two. However, we conjecture that the rates are better for discontinuous functions. 

{As a final remark, the main goal of the paper is to introduce a new tool 
for function approximation with remarkable properties 
in the traditional noiseless setting when exact data is available.
Of course, to validate its potential and efficiency in the more general setting of statistical learning where data can be corrupted by noise (in the $\x$ and/or the value $f(\x)$),
a further detailed analysis is needed but beyond the scope of the present paper.} {We believe that relations to the max-margin support vector machine~\cite{tsochantaridis, bach_book} could facilitate this analysis.}

{
\section{Acknowledgement}
The authors would like to thank Francis Bach for pointing out the links to max-margin support vector machines. The authors would also like to acknowledge the help of Open AI's model o4-mini-high in proving Theorem~\ref{thm:piecewise_cosnt}.}


\begin{thebibliography}{las}

\bibitem{roberto} T. Alamo, R. Tempo, A. Luque, D. R. Ramirez. Randomized methods for design of uncertain systems: Sample complexity and sequential algorithms. Automatica 52:160-172, 2015.

\bibitem{sdp-handbook}
M. Anjos, J.B. Lasserre. (Editors). Handbook on Semidefinite, Conic, and Polynomial Optimization. Springer, New York,  2012.

\bibitem{siam-news-1}
V. Antun, N. M. Gottschling, A. C. Hansen, B. Adcock.
Deep learning in scientific computing: understanding the instability mystery. SIAM News 54:2, 2021.


\bibitem{siam-news-2}
V. Antun, M. J. Colbrook, A. C. Hansen.
Proving existence is not enough: mathematical paradoxes unravel the limits of neural networks in artificial intelligence.
SIAM News 55:4, 2022.


\bibitem{bach_book}
F. Bach. Learning Theory from First Principles. The MIT Press, 2023. \url{https://www.di.ens.fr/~fbach/ltfp_book.pdf}

\bibitem{battles}
Z. Battles, L. N. Trefethen. An extension of Matlab to continuous functions and operators. SIAM J. Sci. Comp.  25(5):1743--1770, 2004.

{
\bibitem{bn01}
A. Ben-Tal, A. Nemirovski. Lectures on modern convex optimization.
MPS-SIAM Series on Optimization, SIAM, 2001.
}

\bibitem{campi} M. C. Campi, S. Garatti. The exact feasibility of randomized solutions of uncertain convex programs. SIAM Journal on Optimization 19(3):1211-1230, 2008.

 \bibitem{campi_personal} M. Campi. Personal communication.
 
 \bibitem{cohen}
A. Cohen, R. DeVore, G. Petrova, P. Wojtaszczyk. Optimal Stable Nonlinear Approximation.
Foundations of Computational Mathematics, 22:607–648, 2022.

\bibitem{chebfun}
T. A. Driscoll, N. Hale, L. N. Trefethen (editors). Chebfun Guide. Pafnuty Publications, 2014.

\bibitem{kernel}
S. De Marchi, F. Marchetti, E. Perracchione. Jumping with variably scaled discontinuous kernels (VSDKs). BIT Numerical Mathematics 60:441--463, 2020


\bibitem{dumitrescu}
B. Dumitrescu. Positive trigonometric polynomials and signal processing applications. Springer, 2007.

\bibitem{google}
P. Florence, C. Lynch, A. Zeng, O. Ramirez, A. Wahid, L. Downs, A. Wong, J. Lee, I. Mordatch, J. Tompson.
Implicit behavioral cloning. Proc. 5th Conf. on Robot Learning, PMLR 164:158-168, 2022.

\bibitem{constructive-2}
D. Henrion, J. B. Lasserre. Graph recovery from incomplete moment information. Constructive Approximation 56:165-187, 2022.

\bibitem{bernardo} B. Llanas, S. Lantar\'on, F. J. S\'ainz. Constructive approximation of discontinuous functions by neural networks. Neural Processing Letters 27:209-226, 2008.

\bibitem{bozzini} M. Bozzini, M. Rossini. The detection and recovery of discontinuity curves from scattered data. J. Computational and Applied Mathematics 240:148-162, 2013.

\bibitem{lasserre01}
J. B. Lasserre. Global optimization with polynomials and the problem of moments. SIAM J.  Optimization 11(3):796-817, 2001.

\bibitem{acta}
J.B. Lasserre.  The Moment SoS Hierarchy: Applications and Related Topics, Acta Numerica 33, pp. 841--908, 2024.

\bibitem{burgers}
S. Marx, T. Weisser, D. Henrion, J. B. Lasserre. A moment approach for entropy solutions to nonlinear hyperbolic PDEs. Mathematical Control and Related Fields, 10(1):13-140, 2020.


\bibitem{constructive-1}
S. Marx, E. Pauwels, T. Weisser, D. Henrion, J. B. Lasserre. Semi-algebraic approximation using Christoffel-Darboux kernel. Constructive Approximation 54:391--429, 2021.

\bibitem{eckhoff} K. S. Eckhoff. Accurate and efficient reconstruction of discontinuous functions from truncated series expansions. Math. Comput. 61(204):745-763, 1993.

\bibitem{f16}
S.-I. Filip. A robust and scalable implementation of the Parks-McClellan algorithm for designing FIR filters. ACM Trans. Math. Software 43(1), 7:1-24, 2016.

\bibitem{yalmip} J. L\"ofberg. YALMIP: A toolbox for modeling and optimization in MATLAB. IEEE Symp. Computer-Aided Control Design (CACSD), Taiwan, 2004.

\bibitem{MOSEK}
MOSEK ApS. {MOSEK} Optimization Toolbox User's Manual. Version 10.0, \url{https://www.mosek.com/}, 2024.
  

\bibitem{py19}
D. Papp, S. Yildiz. Sum-of-squares optimization without semidefinite programming. 
SIAM J. Optim. 29(1):822-851, 2019.

\bibitem{powers} V. Powers, B. Reznick. Polynomials that are positive on an interval. Trans. Amer. Math. Soc. 
352(10):4677-4692, 2000.

\bibitem{WENO}
C.-W. Shu. High order weighted essentially non-oscillatory schemes for convection dominated problems. SIAM Review, 51(1), 82--126, 2009

\bibitem{SeDuMi}
Jos F. Sturm. Using SeDuMi 1.02, a {MATLAB} toolbox for optimization over symmetric cones. Optim. Methods and Softwares. 11-12: 625--653, 1999.

\bibitem{atap}
L. N. Trefethen. Approximation Theory and Aproximation Practice. SIAM, 2013.

\bibitem{poly-kernel}
A. Weisse, G. Wellein, A. Alvermann, H. Fehske. The kernel polynomial method. Reviews of Modern Physics 78(1):275-306, 2006.

\bibitem{tadmor}E. Tadmor. Filters, mollifiers and the computation of the Gibbs phenomenon. Acta Numerica 16:305-378, 2007.

\bibitem{tsochantaridis} I. Tsochantaridis, T. Joachims, T. Hofmann, Y. Altun, Y. Singer. Large margin methods for structured and interdependent output variables. J. Machine Learning Research 6(9):1453-1484, 2005.


\end{thebibliography}
\end{document}